\documentclass[12pt]{article}
\usepackage{amsmath,amsxtra,amssymb,latexsym, amscd,amsthm}
\usepackage{xcolor}
\usepackage{graphicx}
\usepackage{float}
\usepackage{epstopdf}
\usepackage{multirow}
\usepackage{subfigure}

\newtheorem{lemma}{Lemma}
\newtheorem{theorem}{Theorem}

\newtheorem{definition}{Definition}

\theoremstyle{remark}
\newtheorem{remark}{Remark}

\allowdisplaybreaks

\numberwithin{equation}{section}

\begin{document}
{\bf  REGULARIZATION OF  BACKWARD TIME-FRACTIONAL PARABOLIC EQUATIONS
BY SOBOLEV-TYPE EQUATIONS}

{\bf Dinh Nho H\`ao$^1$, Nguyen Van Duc$^2$, Nguyen Van Thang$^2$ and Nguyen Trung Th\`anh$^3$}

{\it$^{1}$ Hanoi Institute of Mathematics, VAST, 18 Hoang Quoc Viet, Hanoi, Vietnam\\Email: hao@math.ac.vn \\
$^{2}$ Department of Mathematics, Vinh University, Vinh City,
Vietnam,\\ Email: nguyenvanducdhv@gmail.com,
nguyenvanthangk17@gmail.com
\\
$^{3}$Department of Mathematics, Rowan University, Glassboro, NJ, USA.\\ Email: nguyent@rowan.edu}

{\small {\bf Abstract:} The problem of determining the initial condition from noisy final observations in time-fractional parabolic equations is considered. This problem is well-known to be ill-posed and it is regularized by backward Sobolev-type equations. Error estimates of H\"{o}lder type are obtained with \textit{a priori}  and \textit{a posteriori} regularization parameter choice rules. The proposed regularization method results in a stable noniterative numerical scheme. The theoretical error estimates are confirmed by numerical tests for one- and two-dimensional equations.}


\section{Introduction}\label{sec:int}

We consider the backward time-fractional parabolic equation:
\begin{equation}\label{main-problem}
\begin{cases}
\dfrac{\partial^\gamma u}{\partial t^\gamma} + Au = 0,\quad 0<t<T,\\
\|u(T)-f\|\le\varepsilon,
\end{cases}
\end{equation}
where $\gamma\in (0,1)$,
\begin{equation}
\dfrac{\partial^\gamma u}{\partial t^\gamma} := \dfrac{1}{\Gamma
(1-\gamma)}\int_0^t(t-s)^{-\gamma}\dfrac{\partial u(\cdot, s)}{\partial
s}ds,
\end{equation}
 is the Caputo derivative \cite{Kilbas,P},  $\Gamma(\cdot)$ is
Euler's Gamma function, and $A: D(A)\subset \mathbb{H}\to
\mathbb{H}$ is a self-adjoint closed operator on a Hilbert space
$\mathbb{H}$. We assume that $-A$ generates a compact contraction
semi-group $\{S(t)\}_{t\geq 0}$ on $\mathbb{H}$ and $A$ admits an
orthonormal eigenbasis $\{\phi_i\}_{i\geq 1}$ in $\mathbb{H}$
associated with the eigenvalues $\{\lambda_i\}_{i\geq 1}$ such that
$$
0<\lambda_1<\lambda_2< \ldots,\ \ \mbox{and}\ \
\lim\limits_{i\to+\infty}\lambda_i=+\infty.
$$
In this work we denote the inner product and norm in $\mathbb{H}$ by
$\left<\cdot,\cdot\right>$ and $\|\cdot\|$, respectively.

As an example of the operator $A$, consider the case
$\mathbb{H}=L^2(\Omega)$, where $\Omega$ is a bounded domain in
$\mathbb{R}^d$, $d \ge 1$, with sufficiently smooth boundary
$\partial\Omega$. As usual, we denote by $H^p(\Omega)$ the Sobolev
spaces.  Denote by $H^1_0(\Omega) = \{u\in  H^1(\Omega):
u|_{\partial\Omega} = 0\}.$ Then $A$ can be chosen as:
\begin{align*}
(Au)(x) :=-\sum\limits_{i=1}^d\frac{\partial}{\partial x_i}
\left(\sum\limits_{j=1}^d a_{ij}(x)\frac{\partial}{\partial
x_j}u(x)\right)+c(x)u(x), x\in \Omega,\label{dn1}
\end{align*}
with $D(A) = H^2(\Omega)\cap H^1_0(\Omega).$ Here, we assume that $a_{ij} = a_{ji} \in C^1(\overline{\Omega}), i,j = 1, \ldots, d;\ c \in C(\overline{\Omega}),\ c(x) \geq 0,\ \forall x \in \overline{\Omega}$; and $\sum_{i,j=1}^d a_{ij} (x) \xi_i\xi_j \geq \nu \sum_{i=1}^d\xi_i^2,\  \forall x \in \overline{\Omega}$ for some $\nu > 0$.

The backward problem (\ref{main-problem}) arises from several practical contexts, for example, in the determination of contaminant sources in underground fluid flow \cite{jin}. It  is well-posed for $t
> 0$ and ill-posed for $t=0$, see \cite{hdt}.  Since the first work
\cite{Liu-Yamamoto-AA} devoted to the backward time-fractional
diffusion equation, several papers on backward time-fractional
parabolic equations have been published. For the mollification
method, see \cite{Liyan Wang-Jijun Liu1,yang-liu2015}; for the
non-local boundary value problem method, see
\cite{hdt,jwang,twei,Ming Yang- Jijun Liu-ANM}; and for Tikhonov
regularization, we refer the reader to \cite{mh,Liyan Wang-Jijun
Liu,jwang2,jwang1}. In this paper we will regularize (\ref{main-problem}) by the backward Sobolev-type equations. {Namely, we approximate the solution $u$
of (\ref{main-problem}) by the solution $v_\alpha$ of the following problem for the Sobolev-type equation:
\begin{equation}\label{zz1ax}
\begin{cases}
\dfrac{\partial^\gamma  v_\alpha}{\partial t^\gamma}  +\alpha A^b\dfrac{\partial^\gamma v_\alpha}{\partial t^\gamma} + A v_\alpha=0,~0<t<T,\\
v_\alpha(T)= f,
\end{cases}
\end{equation}
where $b $ is a real number such that $b\ge 1$, $\alpha>0$ is a regularization parameter, and operator $A^b$ is defined by
	\begin{equation}\label{Ab}
A^b v :=\sum\limits_{n=1}^\infty\lambda_n^b\left<v,\phi_n\right>\phi_n,\quad \forall v\in D(A^b).
\end{equation}
For this operator $A^b$, $I + \alpha  A^b$ is invertible.
Assuming that the fractional differentiation and $A^b$ are interchangeable, we can convert  \eqref{zz1ax} into the following problem
\begin{equation}\label{zz1}
\begin{cases}
\dfrac{\partial^\gamma v_\alpha}{\partial t^\gamma}+A_\alpha v_\alpha=0,~0<t<T,\\
v_\alpha(T)=f,
\end{cases}
\end{equation}
where $A_\alpha =  (I + \alpha A^b)^{-1} A$.
}

Here, for $p > 0$, we define
$$
D(A^p) :=\left\{\psi\in
\mathbb{H}:\sum\limits_{n=1}^\infty\lambda_n^{2p}\left<\psi,
\phi_n\right>^2<\infty\right\}
$$
and the norm
$$
\|\psi\|_p:= \left(\sum\limits_{n=1}^\infty\lambda_n^{2p}\left<\psi,
\phi_n\right>^2\right)^{\frac{1}{2}},~\psi\in D(A^p).
$$

For operator $A$ defined above, it has been shown that $D(A^p)\subset
H^{2p}(\Omega),\ p>0$, and $D(A^{\frac{1}{2}})=H_0^1(\Omega)$, see, e.g.,  \cite{Sakamoto-Yamamoto-Jmaa}.

The regularization method by Sobolev equations was first introduced by Gajewski and
Zacharias for  parabolic equations backward in time \cite{gaj} (i.e., for the case $\gamma = 1$). Since then, the method has been investigated in a number of  works, see, e.g.,
\cite{ewing,fury,huang,long94,long96,padon90,padon,renardy,sho}. However, there has been no error estimate in 
\cite{gaj},  \cite{huang}, 
\cite{sho}, and in the other
related works, error estimates have been obtained only for \textit{a
priori} regularization parameter choice rules. 
\textit{A posteriori} parameter choice rules have not been investigated yet.

To the best of our
knowledge, the present work is the first one to use the Sobolev-type equations for regularizing
the  backward time-fractional parabolic equations
(\ref{main-problem}) with $\gamma\in (0,1)$. {The results of this paper include the following. First, we prove that problem \eqref{zz1} is well-posed (see Lemma \ref{lemmasolution}). Second, we prove
error estimates of H\"{o}lder type for both \textit{a priori} and
\textit{a posteriori} rules for choosing the regularization parameter $\alpha$. The convergence of the regularized solution $v_\alpha$ to the exact solution $u$ is of optimal order}. Indeed, for our \textit{a priori} parameter choice rule, we obtain a convergence rate of order $\frac{p}{p+1}$ for $0 < p < b$ and of order $\frac{b}{b+1}$ for $p\ge b$ (see Theorem \ref{Theorem2}). For our \textit{a posteriori} parameter choice rule, we obtain a convergence rate of order $\frac{p}{p+1}$ for $0 < p < b-1$ and of order $\frac{b}{b+1}$ for $p\ge b -1$ (see Theorem \ref{Theorem3}). This order of convergence was proved to be optimal in H\`ao et~al.~\cite{hdt}, in which the authors also obtained error estimates of the same order as in this paper using the method of non-local boundary value problem.  Third, we propose a numerical method for solving problem \eqref{zz1} using the method of separation of variables and demonstrate the performance of the proposed regularization method using numerical tests in one and two dimensions. We should mention that no other methods for solving backward fractional-order Sobolev equations have been reported. For numerical methods for forward fractional-order Sobolev-type equations, see, e.g., \cite{Chadha:2018}.

To obtain high convergence rates, we choose $A_\alpha$ in a different way compared to that of Gajewski and Zacharias \cite{gaj} and the above papers. More precisely,  we choose
$A_\alpha :=(I+\alpha A^b)^{-1} A$, where $b\geq 1$ is an arbitrary real number, whereas Gajewski and Zacharias \cite{gaj} and the related authors considered $A_\alpha :=(I+\alpha A)^{-1} A$. The error estimates presented in Section \ref{regularization} show that the convergence rate is higher when $b$ is larger.

Compared to other methods for solving problem (\ref{main-problem}), the order of our error estimates is higher than those of \cite{Liyan Wang-Jijun Liu1, jwang2,jwang1,jwang, twei, Ming Yang- Jijun Liu-ANM}, where the authors have used another regularization methods. Indeed, as pointed out in Remarks \ref{remark2} and \ref{remark4}, the order of our error estimates  is larger than $2/3$ for appropriate values of $p$ and $b$, whereas the order of the error estimate
in \cite{jwang, Ming Yang- Jijun Liu-ANM} is not greater than
${1}/{2}$ while that in
\cite{jwang2,jwang1,twei} is not greater than ${2}/{3}$ for all
$p>0$ for their \textit{a priori} parameter choice rule and is not greater than
${1}/{2}$ for their \textit{a posteriori} parameter choice rule.


The paper is organized as follow: in Section \ref{sec:aux} we recall
some basic definitions and present simple inequalities which are
needed for proving the main results in this paper. In Section
\ref{regularization} we describe our regularization method with the
error estimates, the proofs of which will be given in Section
\ref{proofs}. Numerical implementation of
the proposed regularization scheme and numerical tests are presented in Section \ref{sec:num}. Finally, some conclusions are given in Section \ref{sec:conclusion}.

\section{Auxiliary results} \label{sec:aux}
Denote by $E_{\gamma, \beta}(z)$ the Mittag-Leffler function
\cite{Kilbas, P}:
\begin{align}\label{Mittag-Leffler function}
E_{\gamma, \beta}(z)=\sum\limits_{k=0}^\infty\dfrac{z^k}{\Gamma(k\gamma+\beta)}, z \in \mathbb{C}, \gamma>0, ~\beta\in\mathbb{R}.
\end{align}

\begin{definition}{\rm\cite{hdt}} For $u_0\in \mathbb{H}$, a function $u(t): [0,T] \to \mathbb{H}$ is called a solution to problem
\begin{equation}\label{main-problem000}
\begin{cases}
\dfrac{\partial^\gamma u}{\partial t^\gamma}+Au=0,\quad 0<t<T,\\
u(0)=u_0,
\end{cases}
\end{equation}
 if $u(t)\in C^1((0,T),\mathbb{H})\cap C([0,T],\mathbb{H})$,
$u(t)\in D(A)$ for all $t\in(0,T)$, and \eqref{main-problem000}
holds.
\end{definition}

\begin{theorem}\label{0Theorem1}{\rm\cite{hdt}}
Problem \eqref{main-problem000} admits a unique solution, which can
be represented in the form:
\begin{align}\label{congthucnghiem}
u(t)=\sum\limits_{n=1}^\infty
E_{\gamma,
1}\left(-\lambda_nt^\gamma\right)\left<u_0,\phi_n\right>\phi_n.
\end{align}
\end{theorem}

\begin{lemma}\label{Young's inequality} {\rm (Young's inequality)}
If $a,b$  are nonnegative numbers and $m,n$ are positive numbers
such that $\dfrac{1}{m}+\dfrac{1}{n}=1$, then $ ab \le
\dfrac{a^m}{m}+ \dfrac{b^n}{n}. $
\end{lemma}

\begin{lemma}\label{lemmamoi} {\rm\cite{jwang}} For any $\lambda_n$
satisfying $\lambda_n \geq \lambda_1>0$, there exist positive
constants $\overline{C}_1$ and $\overline{C}_2$ depending on $ \gamma,\ T,\
\lambda_1$ such that
\begin{align*}
\dfrac{\overline{C}_1}{\lambda_n}\le
E_{\gamma,1}(-\lambda_nT^\gamma)\le\dfrac{\overline{C}_2}{\lambda_n}.
\end{align*}
\end{lemma}
\begin{lemma}\label{lemmamoi01} {\rm\cite{Liu-Yamamoto-AA}} There exist positive
constants $\overline{C}_3,\overline{C}_4$ and $\overline{C}_5$
depending on $ \gamma$ such that
\begin{align*}
&a)\quad
\dfrac{\overline{C}_3}{\Gamma(1-\gamma)}.\dfrac{1}{1-x}\le
E_{\gamma,1}(x)\le\dfrac{\overline{C}_4}{\Gamma(1-\gamma)}.\dfrac{1}{1-x},~\text{for
all}~x\le 0.
\\
&b)\quad
|E_{\gamma,0}(x)|\le\dfrac{\overline{C}_5}{\Gamma(-\gamma)}\dfrac{1}{1-x},~\mbox{for
all}~x\leq0.
\end{align*}
\end{lemma}

\begin{lemma}\label{0lemma1}{\rm(\cite{Liu-Yamamoto-AA}, p. 1779)}
Let $\gamma\in (0,1)$ and $t>0$. We have
\begin{align}
 \dfrac{d}{ds} E_{\gamma, 1}(s t^\gamma)
 =\frac{1}{s t^{\gamma}}E_{\gamma, 0}(s t^\gamma),~\mbox{for all}\quad s\in\mathbb{ R},~s\neq
 0.
\end{align}
\end{lemma}

%
%

\begin{lemma} For simplicity of notation, we denote $B_\alpha := (I + \alpha A^b)^{-1}$. We have the following representations for $A_\alpha $ and $B_\alpha$:
	\begin{align}
    &A_\alpha v :=\sum\limits_{n=1}^\infty\left(\dfrac{\lambda_n}{1+\alpha\lambda_n^b}\right)
\left<v,\phi_n\right>\phi_n, \label{Aalpha} \\
	&B_\alpha v :=\sum\limits_{n=1}^\infty\left(\dfrac{1}{1+\alpha\lambda_n^b}\right)\left<v,\phi_n\right>\phi_n. \label{Balpha}
\end{align}
\end{lemma}

\section{Regularization methods and error estimates} \label{regularization}
In this section, we regularize problem (\ref{main-problem}) by the backward Sobolev-type equation (\ref{zz1}) and
propose \textit{a priori} and  \textit{a posteriori} methods for choosing the
regularization parameter $\alpha$ which yield error estimates of
H\"older type. 

\subsection{\textit{A priori} parameter choice rule}
\begin{theorem}\label{Theorem2} For $b\geq 1$,
 problem (\ref{zz1}) is well-posed. Moreover, if the solution $u(t)$ of problem (\ref{main-problem}) satisfies
\begin{align}\label{constraint2sm}
\|u(0)\|_p\le E, \quad p>0,\ E>\varepsilon,
\end{align}
then the following statements hold:
\par(i) If $0<p<b$, then with
$\alpha=\left(\dfrac{\varepsilon}{E}\right)^{\frac{b}{p+1}}$, there
exists a  constant $C_1$ such that
\begin{equation*}
\|u(0)-v_\alpha(0)\|\le
C_1\varepsilon^{\frac{p}{p+1}}E^{\frac{1}{p+1}}.
\end{equation*}

\par(ii) If $p \geq b$, then with
$\alpha=\left(\dfrac{\varepsilon}{E}\right)^{\frac{b}{b+1}}$, there
exists  a constant $C_2$ such that
\begin{equation*}
\|u(0)-v_\alpha(0)\|\le C_2
\varepsilon^{\frac{b}{b+1}}E^{\frac{1}{b+1}}.
\end{equation*}

\end{theorem}
\begin{remark}\label{remark2}
We note that $\frac{p}{p+1}>\frac{2}{3}$ when $p>2$ and
$\frac{b}{b+1}>\frac{2}{3}$ when $b>2$. Therefore, the order of our
error estimates is greater than $\frac{2}{3}$ when $2<p<b$ or
$p\geq b>2$.
\end{remark}

\subsection{\textit{A posteriori} parameter choice rule}
\label{section-a-posteriori}

\begin{theorem} \label{Theorem3}
Let $b>1$ and $B_{\alpha}$ be given by \eqref{Balpha}. Assume
that $0<\varepsilon < \|f\|$ and $\tau>1$ is a constant satisfying
$0<\tau\varepsilon\le\|f\|$. Then, there exists a unique
number $\alpha_\varepsilon>0$ such that
\begin{equation}
\|B_{\alpha_\varepsilon}f-f\|=\tau\varepsilon. \label{a-posteriori}
\end{equation}
Let  $u(t)$ and $v_{\alpha_{\epsilon}}$ be the solutions of \eqref{main-problem} and \eqref{zz1}, respectively. Then for $\varepsilon$ small enough there exist  constants $C_3$ and $C_4 $ such
that
\begin{itemize} \item[i)]
\begin{equation*}
\|u(0)-v_{\alpha_\varepsilon}(0)\|\le
C_3\varepsilon^{\frac{p}{p+1}}\| u(0)\|_p^{\frac{1}{p+1}} \text{ for } p < b-1.
\end{equation*}
\item [ii)] 
\begin{equation*}
\|u(0)-v_{\alpha_\varepsilon}(0)\|\le
C_4\varepsilon^{\frac{p}{p+1}}\|u(0)\|_p^{\frac{1}{p+1}}
+\varepsilon^{\frac{b-1}{b}}\|u(0)\|_p^{\frac{1}{b}}, \text{ for } p \geq b-1.
\end{equation*}
\end{itemize}
\end{theorem}

\bigskip

\begin{remark} In case (i) of Theorems \ref{Theorem2} and \ref{Theorem3},  the convergence rate
$E^{\frac{1}{p+1}}\delta^{\frac{p}{p+1}}$  is of optimal order as pointed out in \cite{hdt}.
\end{remark}

\begin{remark}\label{remark4}
Since $\frac{p}{p+1}>\frac{2}{3}$ when $p>2$ and
$\frac{b-1}{b}>\frac{2}{3}$ when $b>3$, the order of our
error estimates in Theorem \ref{Theorem3} is greater than $\frac{2}{3}$ when $2<p<b-1$ or
$p\geq b>3$.
\end{remark}

\begin{remark}The authors of
\cite{ewing,fury,huang,long94,long96,padon90,padon,renardy,sho}
used the Sobolev equations to regularize backward
parabolic equations, i.e., the case $\gamma=1$. The results were obtained only for the \textit{a priori} parameter choice rule and with  $b=1$. Here we not only obtain the optimal convergence rates for the backward time-fractional  equation (\ref{main-problem})  for both \textit{a priori} and \textit{a posteriori}
parameter choice rules, but also with an  arbitrary positive constant $b\geq 1$.
\end{remark}

\section{Proofs of the main results}\label{proofs}
\subsection{Proof of Theorem \ref{Theorem2} } \label{proof-th2}
First, we present some auxiliary results.
\begin{lemma}\label{lemmasolution}Problem (\ref{zz1}) admits a unique solution
\begin{align}\label{regularization-solution}
v_\alpha(t)=\sum\limits_{n=1}^\infty\dfrac{E_{\gamma,1}\left(-\lambda_{\alpha
n} t^\gamma\right)\left<f,\phi_n\right>\phi_n} {
E_{\gamma,1}\left(-\lambda_{\alpha n} T^\gamma\right)},\ \forall
t\in[0,T],
\end{align}
where $\lambda_{\alpha n}=\dfrac{\lambda_n}{1+\alpha\lambda_n^b}$.
Furthermore, there exists a constant $\overline{C}_6$ such that
$$
\|v_\alpha(t)\|\le
\overline{C}_6(1+\alpha^{-\frac{1}{b}})\|f\|,\ \forall t\in[0,T].
$$
\end{lemma}

\begin{proof}
Formula (\ref{regularization-solution}) is obtained by direct
calculations. From inequalities $0\leq
E_{\gamma,1}\left(-\lambda_{\alpha n}
 T^\gamma\right)\leq 1$, it follows that
\begin{align*}
\|v_\alpha(t)\|^2&\leq\sum\limits_{n=1}^\infty\dfrac{\left<f,\phi_n\right>^2}
{ \left[E_{\gamma,1}\left(-\lambda_{\alpha n}
T^\gamma\right)\right]^2}.
\end{align*}
From Lemma~\ref{lemmamoi01}, we obtain
\begin{align}
\|v_\alpha(t)\|^2&\leq\sum\limits_{n=1}^\infty\left(\dfrac{\left<f,\phi_n\right>}
{ \frac{C}{\Gamma(1-\gamma)}\frac{1}{1+\lambda_{\alpha n}
T^\gamma}}\right)^2\nonumber \\
&\leq\left(\frac{\Gamma(1-\gamma)}{C}\right)^2\sum\limits_{n=1}^\infty(1+\lambda_{\alpha
n} T^\gamma)^2\left<f,\phi_n\right>^2.\label{cb0}
\end{align}
For $b>1$, we have  $$1+\alpha\lambda_{n}^b\geq
\frac{b-1}{b}.1^{\frac{b}{b-1}}+\frac{1}{b}\left(\alpha^{1/b}\lambda_{n}\right)^b\geq\alpha^{1/b}\lambda_{n},$$
 or
\begin{align*}\label{cb1}
1+\alpha\lambda_{n}^b\geq
\alpha^{1/b}\lambda_{n}~\text{for all} ~b\geq1.
\end{align*}
Therefore
\begin{align*}\lambda_{\alpha
n}:=\dfrac{\lambda_n}{1+\alpha\lambda_n^b} \leq \alpha^{-1/b}.
\end{align*}
It follows from this inequality and (\ref{cb0}) that there exists a
constant $\overline{C}_{7}>0$ such that
\begin{align*}\|v_\alpha(t)\|^2\leq
\overline{C}_{7}(1+\alpha^{-1/b})^2\|f\|^2.
\end{align*}
 The
lemma is proved.
\end{proof}

In the following, we denote by $w_\alpha(t)$ the solution of the problem
\begin{equation}\label{regularization-problema1}
\begin{cases}
\dfrac{\partial^\gamma w_\alpha}{\partial t^\gamma}+A_\alpha w_\alpha=0,\quad 0<t<T,\\
w_\alpha(T)=u(T),
\end{cases}
\end{equation}
\begin{lemma}\label{Lemma7a1} If $w_\alpha(t)$ is the solution of problem
\eqref{regularization-problema1} and $v_\alpha(t)$ is the solution of
problem \eqref{zz1}, then
$$\|v_\alpha(0)-w_\alpha(0)\|\leq \overline{C}_6(1+\alpha^{-1/b})\varepsilon.$$
\end{lemma}
\begin{proof} We see that $w_\alpha(t)-v_\alpha(t)$ solves  problem \eqref{zz1}
 with  $f$ being replaced by
$u(T)-f$. Using Lemma~\ref{lemmasolution}, we have
$$\|v_\alpha(0)-w_\alpha(0)\|\leq \overline{C}_6(1+\alpha^{-1/b})\|u(T)-f\|\leq \overline{C}_6(1+\alpha^{-1/b})\varepsilon.$$The
lemma is proved.
\end{proof}
\begin{lemma}\label{Lemma7}
If $\|u(0)\|_p\le E$ for some positive constants
$p, \ E>0$, then there exist constants $\overline{C}_{8}$ and
$\overline{C}_9$ such that
\begin{equation*}
\|u(0)-w_\alpha(0)\|^2\le
\begin{cases}
\overline{C}_8\alpha^{2p/b}E^2  &\text{if} \quad~ p< b,\\
\overline{C}_9(\alpha^{2}E^2+\alpha^{2p/b}E^2) &\text{if} \quad~
p\geq b.
\end{cases}
\end{equation*}
\end{lemma}

\begin{proof} We have
\begin{align}\label{t2}
\|u(0)-w_\alpha(0)\|^2&=\sum\limits_{n=1}^\infty\left<u(0)-w_\alpha(0),\phi_n\right>^2\notag\\
&=
\sum\limits_{n=1}^\infty
\left(\left<u(0),\phi_n\right>-\dfrac{\left<u(T),\phi_n\right>}{E_{\gamma,1}\left(-\lambda_{\alpha n} T^\gamma\right)}\right)^2\notag\\
&=\sum\limits_{n=1}^\infty\left(\left<u(0),\phi_n\right>-
\dfrac{E_{\gamma,1}(-\lambda_nT^\gamma)\left<u(0),\phi_n\right>}{E_{\gamma,1}\left(-\lambda_{\alpha
n} T^\gamma\right)}\right)^2 \notag\\
&=\sum\limits_{n=1}^\infty
\left<u(0),\phi_n\right>^2\left(1-\dfrac{E_{\gamma,1}\left(-\lambda_n
T^\gamma\right)}{E_{\gamma,1}\left(-\lambda_{\alpha n}
T^\gamma\right)}\right)^2 \notag\\
&=\sum\limits_{n=1}^{n_1-1}
\lambda_n^{2p}\left<u(0),\phi_n\right>^2\left(\dfrac{E_{\gamma,1}\left(-\lambda_{\alpha
n} T^\gamma\right)-E_{\gamma,1}\left(-\lambda_n
T^\gamma\right)}{E_{\gamma,1}\left(-\lambda_{\alpha n}
T^\gamma\right)}\right)^2 \lambda_n^{-2p}\notag\\
&\quad+\sum\limits_{n=n_1}^{\infty}
\lambda_n^{2p}\left<u(0),\phi_n\right>^2\left(1-\dfrac{E_{\gamma,1}\left(-\lambda_n
T^\gamma\right)}{E_{\gamma,1}\left(-\lambda_{\alpha n}
T^\gamma\right)}\right)^2 \lambda_n^{-2p},
\end{align}
where $n_1=\min\{n: \lambda_n\geq \alpha^{-1/b}\}$. Let
$h(s)=E_{\gamma,1}\left(s T^\gamma\right),~s<0$. Then $h(s)$ is an
increasing function. Since $-\lambda_n\leq-\lambda_{\alpha n}$, we
have $E_{\gamma,1}\left(-\lambda_n T^\gamma\right)\leq
E_{\gamma,1}\left(-\lambda_{\alpha n} T^\gamma\right).$ Therefore
\begin{align}\label{t3a}
\left(1-\dfrac{E_{\gamma,1}\left(-\lambda_n
T^\gamma\right)}{E_{\gamma,1}\left(-\lambda_{\alpha n}
T^\gamma\right)}\right)^2\leq 1.
\end{align}
Hence,
\begin{align}\label{sbt2}
\sum\limits_{n=n_1}^{\infty}
\lambda_n^{2p}\left<u(0),\phi_n\right>^2\left(1-\dfrac{E_{\gamma,1}\left(-\lambda_n
T^\gamma\right)}{E_{\gamma,1}\left(-\lambda_{\alpha n}
T^\gamma\right)}\right)^2
\lambda_n^{-2p}
&\leq\sum\limits_{n=n_1}^{\infty}
\lambda_n^{2p}\left<u(0),\phi_n\right>^2\lambda_n^{-2p}\notag\\
&\leq\alpha^{2p/b}\sum\limits_{n=n_1}^{\infty}
\lambda_n^{2p}\left<u(0),\phi_n\right>^2\notag\\
&\leq \alpha^{2p/b}E^2.
\end{align}
By Lemma~\ref{0lemma1}, we have
$\dfrac{d}{ds}h(s)=\dfrac{1}{sT^\gamma}E_{\gamma,0}(s T^\gamma)$.
 Therefore, there exist constants $\xi_n\in
(-\lambda_n,-\lambda_{\alpha n})$, $n\ge 1$, such that
\begin{align*}E_{\gamma,1}\left(-\lambda_{\alpha
n} T^\gamma\right)-E_{\gamma,1}\left(-\lambda_n T^\gamma\right)&=
\dfrac{1}{\xi_nT^\gamma}E_{\gamma,0}\left(\xi_n
T^\gamma\right)(\lambda_{\alpha n}-\lambda_n)\notag\\
&=\dfrac{-1}{\xi_n}E_{\gamma,0}\left(\xi_n
T^\gamma\right)\dfrac{\alpha\lambda_{n}^{b+1}}{1+\alpha\lambda_n^b}.
\end{align*}
From Lemma~\ref{lemmamoi01}, there exists a constant
$\overline{C}_{10}>0$ such that
\begin{align}\label{sbt1}|E_{\gamma,1}\left(-\lambda_{\alpha
n} T^\gamma\right)-E_{\gamma,1}\left(-\lambda_n
T^\gamma\right)|&\leq\dfrac{-\overline{C}_{10}}{\xi_n(1-\xi_nT^\gamma)}\dfrac{\alpha\lambda_{n}^{b+1}}{1+\alpha\lambda_n^b}\notag\\
&\leq\dfrac{\overline{C}_{10}}{\xi_n^2T^{\gamma}}\dfrac{\alpha\lambda_{n}^{b+1}}{1+\alpha\lambda_n^b}.
\end{align}
On the other hand, from $\lambda_{n}\leq \alpha^{-1/b},~n\leq n_1-1,$
it follows that $\lambda_{\alpha n}\geq
\frac{\lambda_n}{2}>\frac{\lambda_1}{2}$. From
Lemma~\ref{lemmamoi} it follows that there exists a constant $\overline{C}_{11}>0$
such that
\begin{align*}&\sum\limits_{n=1}^{n_1-1}
\lambda_n^{2p}\left<u(0),\phi_n\right>^2\left(\dfrac{E_{\gamma,1}\left(-\lambda_{\alpha
n} T^\gamma\right)-E_{\gamma,1}\left(-\lambda_n
T^\gamma\right)}{E_{\gamma,1}\left(-\lambda_{\alpha n}
T^\gamma\right)}\right)^2 \lambda_n^{-2p}\notag\\
&\leq\overline{C}_{11}\sum\limits_{n=1}^{n_1-1}
\lambda_n^{2p}\left<u(0),\phi_n\right>^2\left(\dfrac{\dfrac{1}{\xi_n^2T^{\gamma}}
\dfrac{\alpha\lambda_{n}^{b+1}}{1+\alpha\lambda_n^b}}{\frac{1}{\lambda_{\alpha
n}T^\gamma}}\right)^2 \lambda_n^{-2p}\notag\\
&=\overline{C}_{11}\sum\limits_{n=1}^{n_1-1}
\lambda_n^{2p}\left<u(0),\phi_n\right>^2\left(\dfrac{1}{\xi_n^2}
\alpha\lambda_{n}^{b-p}\lambda_{\alpha n}^2\right)^2. \notag\\
\end{align*}

Note that $\xi_n\in (-\lambda_n,-\lambda_{\alpha n})$. It follows
that $\xi_n^2\geq\lambda_{\alpha n}^2$. Therefore
\begin{align}\label{dt}&\sum\limits_{n=1}^{n_1-1}
\lambda_n^{2p}\left<u(0),\phi_n\right>^2\left(\dfrac{E_{\gamma,1}\left(-\lambda_{\alpha
n} T^\gamma\right)-E_{\gamma,1}\left(-\lambda_n
T^\gamma\right)}{E_{\gamma,1}\left(-\lambda_{\alpha n}
T^\gamma\right)}\right)^2 \lambda_n^{-2p}\notag\\
&\leq\overline{C}_{11}\sum\limits_{n=1}^{n_1-1}
\lambda_n^{2p}\left<u(0),\phi_n\right>^2\left(
\alpha\lambda_{n}^{b-p}\right)^2.
\end{align}

If $p< b$, then $\lambda_{n}^{b-p}\leq \alpha^{(p-b)/b},~\text{for
all}~ n\leq n_1-1$. We obtain
\begin{align}\label{sbta1}&\sum\limits_{n=1}^{n_1-1}
\lambda_n^{2p}\left<u(0),\phi_n\right>^2\left(\dfrac{E_{\gamma,1}\left(-\lambda_{\alpha
n} T^\gamma\right)-E_{\gamma,1}\left(-\lambda_n
T^\gamma\right)}{E_{\gamma,1}\left(-\lambda_{\alpha n}
T^\gamma\right)}\right)^2
\lambda_n^{-2p}\notag\\
&\leq\overline{C}_{11}\alpha^{2p/b}\sum\limits_{n=1}^{n_1-1}
\lambda_n^{2p}\left<u(0),\phi_n\right>^2\leq\overline{C}_{11}\alpha^{2p/b}E^2.
\end{align}

If $p\geq b$, then $\lambda_{n}^{b-p}\leq
\lambda_{1}^{b-p},~\text{for all}~ n\leq n_1-1$. There exists a
constant $\overline{C}_{12}>0$ such that
\begin{align}\label{sbta2}&\sum\limits_{n=1}^{n_1-1}
\lambda_n^{2p}\left<u(0),\phi_n\right>^2\left(\dfrac{E_{\gamma,1}\left(-\lambda_{\alpha
n} T^\gamma\right)-E_{\gamma,1}\left(-\lambda_n
T^\gamma\right)}{E_{\gamma,1}\left(-\lambda_{\alpha n}
T^\gamma\right)}\right)^2
\lambda_n^{-2p}\notag\\
&\leq\overline{C}_{12}\alpha^{2}\sum\limits_{n=1}^{n_1-1}
\lambda_n^{2p}\left<u(0),\phi_n\right>^2\leq\overline{C}_{12}\alpha^{2}E^2.
\end{align}

From (\ref{t2}), \eqref{sbt2}, \eqref{sbta1} and (\ref{sbta2})
we obtain the second estimate of the lemma. The proof is complete.
\end{proof}

Now we are in a position to prove Theorem \ref{Theorem2}. We note that the well-posedness of problem (\ref{zz1}) is
implied from Lemma~\ref{lemmasolution}.

{\bf Proof of part (i) of Theorem \ref{Theorem2}.}

If $p< b$, from Lemmas \ref{Lemma7a1} and \ref{Lemma7} there exists
a constant $\overline{C}_{13}>0$ such that
\begin{align*}
\|u(0)-v_\alpha(0)\|&\leq\|u(0)-w_\alpha(0)\|+\|v_\alpha(0)-w_\alpha(0)\|\\
&\le
\overline{C}_{13}\left(\alpha^{p/b}E+\alpha^{-1/b}\varepsilon+\varepsilon\right).
\end{align*}
 Choosing
$\alpha=\left(\dfrac{\varepsilon}{E}\right)^{\frac{b}{p+1}}$, we
have
\begin{align*}
\|u(0)-v_\alpha(0)\| &\le
\overline{C}_{13}\left(2\varepsilon^{\frac{p}{p+1}}E^{\frac{1}{p+1}}+\varepsilon\right).
\end{align*}
For $E>\varepsilon$, we have
$\varepsilon\leq\varepsilon^{\frac{p}{p+1}}E^{\frac{1}{p+1}}$. Hence, part (i) of Theorem \ref{Theorem2} is proved.

{\bf Proof of part (ii) of Theorem \ref{Theorem2}.}

If $p\geq b$, from Lemma~\ref{Lemma7a1} and Lemma~\ref{Lemma7} there
exists a constant $\overline{C}_{14}>0$ such that
\begin{align*}
\|u(0)-v_\alpha(0)\|&\leq\|u(0)-w_\alpha(0)\|+\|v_\alpha(0)-w_\alpha(0)\|\\
&\le \overline{C}_{14}(\alpha
E+\alpha^{p/b}E+\alpha^{-1/b}\varepsilon+\varepsilon).
\end{align*}
Choosing
$\alpha=\left(\dfrac{\varepsilon}{E}\right)^{\frac{b}{b+1}}$, we
have
\begin{align*}
\|u(0)-v_\alpha(0)\|\le
\overline{C}_{14}(2\varepsilon^{\frac{b}{b+1}}E^{\frac{1}{b+1}}+\varepsilon^{\frac{p}{b+1}}E^{1-\frac{p}{b+1}}+\varepsilon).
\end{align*}
For $E>\varepsilon$ and $p\ge b$, we have
$\varepsilon\leq\varepsilon^{\frac{b}{b+1}}E^{\frac{1}{b+1}}$ and
$\varepsilon^{\frac{p}{b+1}}E^{1-\frac{p}{b+1}}\leq
\varepsilon^{\frac{b}{b+1}}E^{\frac{1}{b+1}}$. Hence,  part (ii) of Theorem \ref{Theorem2} is proved. Therefore, the proof of Theorem \ref{Theorem2} is complete.

\subsection{ Proof of Theorem \ref{Theorem3} } \label{proof-th3}
First, we prove the following lemma.

\begin{lemma} \label{Lemma8}
Set $\rho (\alpha) :=\|B_\alpha f-f\|$ and suppose that $f \neq 0$.
Then
\par a) $\rho $ is a continuous function,
\par b) $\lim\limits_{\alpha\to 0^{+}}\rho (\alpha)=0,$
\par c) $\lim\limits_{\alpha\to +\infty}\rho (\alpha)=\|f\|,$
\par d) $\rho $ is a strictly increasing function.
\end{lemma}

\begin{proof} a) From \eqref{Balpha} we have
\begin{equation}\label{lemma8estimate1}
\rho^2(\alpha)=\|B_\alpha f-f\|^2
=\sum\limits_{n=1}^\infty\left(\dfrac{\left<f,\phi_n\right>}
{1+\alpha\lambda_n^b}-\left<f,\phi_n\right>\right)^2=\sum\limits_{n=1}^\infty\dfrac{\alpha^2\lambda_n^{2b}\left<f,\phi_n\right>^2}
{(1+\alpha\lambda_n^b)^2}.
\end{equation}
For $\alpha_0>0$, we have
\begin{align*}|\rho^2(\alpha)-\rho^2(\alpha_0)|&=\left|\sum\limits_{n=1}^\infty\left<f,\phi_n\right>^2
\left(\dfrac{\alpha}{1+\alpha\lambda_n^b}+\dfrac{\alpha_0}{1+\alpha_0\lambda_n^b}\right)
\dfrac{(\alpha-\alpha_0)\lambda_n^{2b}}{(1+\alpha\lambda_n^b)(1+\alpha_0\lambda_n^b)}\right|\\
&\leq\dfrac{|\alpha-\alpha_0|(\alpha+\alpha_0)}{\alpha\alpha_0}\sum\limits_{n=1}^\infty\left<f,\phi_n\right>^2
=\dfrac{|\alpha-\alpha_0|(\alpha+\alpha_0)}{\alpha\alpha_0}\|f\|^2.
\end{align*}
Therefore, $\rho $ is a continuous function.

\par\noindent
b) Let $\delta$ be an arbitrary positive number. Since
$\|f\|^2=\sum\limits_{n=1}^\infty\left<f,\phi_n\right>^2$, there
exists a positive integer $n_\delta$ such that
$\sum\limits_{n=n_\delta+1}^\infty\left<f,\phi_n\right>^2<\dfrac{\delta^2}{2}$.
For $0<\alpha<\dfrac{\delta}{\sqrt{2}\lambda_{n_\delta}^{b}\|f\|}$,
 we have
\begin{align*}
\rho^2(\alpha)&\le
\sum\limits_{n=1}^{n_\delta}\dfrac{\alpha^2\lambda_n^{2b}\left<f,\phi_n\right>^2}
{(1+\alpha\lambda_n^b)^2} +\sum\limits_{n=n_\delta+1}^\infty\left<f,
\phi_n\right>^2 \\
& \leq
\alpha^2\lambda_{n_\delta}^{2b}\sum\limits_{n=1}^{n_\delta} \left<f,
\phi_n\right>^2 +
\dfrac{\delta^2}{2}\\
&\leq\alpha^2\lambda_{n_\delta}^{2b}\|f\|^2+\dfrac{\delta^2}{2}\leq
\delta^2.
\end{align*}
This implies that  $\lim\limits_{\alpha\to 0^{+}}\rho (\alpha)=0.$
\par\noindent
c) From (\ref{lemma8estimate1}) we have $\rho(\alpha)<\|f\|$. Since $\lambda_n\geq \lambda_1$ for all $n\geq 1$,   (\ref{lemma8estimate1}) also implies that
\begin{align*}
\rho^2(\alpha) \geq
\sum\limits_{n=1}^\infty\dfrac{\alpha^2\lambda_1^{2b}\left<f,\phi_n\right>^2}
{(1+\alpha\lambda_1^b)^2}.
\end{align*}
Therefore, $ \|f\|\geq
\rho(\alpha)\geq\dfrac{\alpha^2\lambda_1^{2b}\|f\|}
{(1+\alpha\lambda_1^b)^2}. $ This implies that
$\lim\limits_{\alpha\to +\infty}\rho (\alpha)=\|f\|.$
\par\noindent
d) Since the function $\dfrac{\alpha \lambda_n^b}{1 + \alpha\lambda_n^b}$ is strictly increasing with respect to $\alpha > 0$, it follows from (\ref{lemma8estimate1}) that $\rho(\alpha)$ is strictly increasing if there exists a positive integer $n$ such that $\left<f,
\phi_{n}\right>^2>0$. This condition is true since $\|f\| > 0$.
The lemma is proved.
\end{proof}

Now we are in a position to prove Theorem \ref{Theorem3}.

{\bf Proof of part (i) of Theorem \ref{Theorem3}.}

It follows from Lemma \ref{Lemma8} that there exists a unique number
$\alpha_\varepsilon > 0$ satisfying \eqref{a-posteriori}. From \eqref{t2}, \eqref{t3a} and \eqref{dt},  there exists a constant
$\overline{C}_{15}>0$ such that
\begin{align}\label{sbl1}\|u(0)-w_{\alpha_\varepsilon}(0)\|^2&\leq\overline{C}_{15}\left(\sum\limits_{n=1}^{n_2-1}
\lambda_n^{2p}\left<u(0),\phi_n\right>^2\left(\alpha_\varepsilon
\lambda_n^{b-p}\right)^2+\sum\limits_{n=n_2}^\infty
\left<u(0),\phi_n\right>^2\right)\notag\\
 &\leq\overline{C}_{15}\left(\sum\limits_{n=1}^{n_2-1}
\left<u(0),\phi_n\right>^2\alpha_\varepsilon^2
\lambda_n^{2b}+\sum\limits_{n=n_2}^\infty
\left<u(0),\phi_n\right>^2\right),
\end{align}
where $n_2=\min\{n: \lambda_n\geq \alpha_\varepsilon^{-1/b}\}$. For
$n\leq n_2-1$, we have $\lambda_n< \alpha_\varepsilon^{-1/b}$. Using the
H\"{o}lder inequality, we obtain
\begin{align}\label{sbl1a}
\sum\limits_{n=1}^{n_2-1}
\left<u(0),\phi_n\right>^2\alpha_\varepsilon^2 \lambda_n^{2b}&\leq
\sum\limits_{n=1}^{n_2-1}
\left<u(0),\phi_n\right>^2\alpha_\varepsilon^{2\frac{p}{p+1}}\lambda_n^{\frac{2pb}{p+1}}\notag\\
&=\sum\limits_{n=1}^{n_2-1}
\left(|\left<u(0),\phi_n\right>|^{\frac{2}{p+1}}\lambda_n^{\frac{2p}{p+1}}\right)
\left(|\left<u(0),\phi_n\right>|^{\frac{2p}{p+1}}\alpha_\varepsilon^{\frac{2p}{p+1}}\lambda_n^{\frac{2p(b-1)}{p+1}}\right)\notag\\
&\leq\left(\sum\limits_{n=1}^{n_2-1}\left<u(0),\phi_n\right>^2\lambda_n^{2p}\right)^{\frac{1}{p+1}}
\left(\sum\limits_{n=1}^{n_2-1}
\left<u(0),\phi_n\right>^2\alpha_\varepsilon^2
\lambda_n^{2(b-1)}\right)^{\frac{p}{p+1}}\notag\\
&\leq \| u(0)\|_p^{\frac{2}{p+1}}\left(\sum\limits_{n=1}^{n_2-1}
\left<u(0),\phi_n\right>^2\alpha_\varepsilon^2
\lambda_n^{2(b-1)}\right)^{\frac{p}{p+1}}.
\end{align}
On the other hand, since $\lambda_n \geq
\alpha_\varepsilon^{-1/b}$ for $n\geq n_2$, we have
\begin{align}\label{sbl1b}
\sum\limits_{n=n_2}^\infty
\left<u(0),\phi_n\right>^2&=\sum\limits_{n=n_2}^\infty
\left(|\left<u(0),\phi_n\right>|^{\frac{2}{p+1}}\lambda_n^{\frac{2p}{p+1}}\right)
\left(|\left<u(0),\phi_n\right>|^{\frac{2p}{p+1}}\lambda_n^{\frac{-2p}{p+1}}\right)\notag\\
&\leq\left(\sum\limits_{n=n_2}^\infty\left<u(0),\phi_n\right>^2\lambda_n^{2p}\right)^{\frac{1}{p+1}}
\left(\sum\limits_{n=n_2}^\infty\left<u(0),\phi_n\right>^2\lambda_n^{-2}\right)^{\frac{p}{p+1}}\notag\\
&\leq
\| u(0)\|_p^{\frac{2}{p+1}}\left(\sum\limits_{n=n_2}^\infty\left<u(0),\phi_n\right>^2\lambda_n^{-2}\right)^{\frac{p}{p+1}}.
\end{align}
 Using Lemma~\ref{lemmamoi},
we have
\begin{align*}
\tau\varepsilon&=\|B_{\alpha_\varepsilon}f-f\| =
\left\|\sum\limits_{n=1}^\infty\dfrac{\alpha_\varepsilon\lambda_n^b\left<f,\phi_n\right>\phi_n}
{1+\alpha_\varepsilon\lambda_n^b}\right\|\notag\\
&=
\left\|\sum\limits_{n=1}^\infty\dfrac{\alpha_\varepsilon\lambda_n^b\left<u(T),\phi_n\right>\phi_n}
{1+\alpha_\varepsilon\lambda_n^k}-
\sum\limits_{n=1}^\infty\dfrac{\alpha_\varepsilon\lambda_n^b\left<u(T)-f,\phi_n\right>\phi_n}
{1+\alpha_\varepsilon\lambda_n^b}\right\|\notag\\
&\geq
 \left\|\sum\limits_{n=1}^\infty\dfrac{\alpha_\varepsilon\lambda_n^bE_{\gamma,1}
 \left(-\lambda_n T^\gamma\right)\left<u(0),\phi_n\right>\phi_n}
{1+\alpha_\varepsilon\lambda_n^b}\right\| -
\left\|\sum\limits_{n=1}^\infty\left<u(T)-f,\phi_n\right>\phi_n\right\|\notag\\
&\geq \overline{C}_1
\left\|\sum\limits_{n=1}^\infty\dfrac{\alpha_\varepsilon\lambda_n^{b-1}\left<u(0),\phi_n\right>\phi_n}
{1+\alpha_\varepsilon\lambda_n^b}\right\| - \varepsilon.
\end{align*}
Using $\lambda_n< \alpha_\varepsilon^{-1/b}$ for all
 $n\leq n_2-1$ again, we obtain
\begin{align}\label{ac1}
(\tau+1)\varepsilon&\geq \overline{C}_1
\left\|\sum\limits_{n=1}^{n_2-1}\dfrac{\alpha_\varepsilon\lambda_n^{b-1}\left<u(0),\phi_n\right>\phi_n}
{1+\alpha_\varepsilon\lambda_n^b}\right\|\notag\\
&\geq\dfrac{\overline{C}_1}{2}
\left\|\sum\limits_{n=1}^{n_2-1}\alpha_\varepsilon\lambda_n^{b-1}\left<u(0),\phi_n\right>\phi_n\right\|.
\end{align}
Similarly, for
 $n\geq n_2$, we have
\begin{align}\label{ac2}
(\tau+1)\varepsilon&\geq \overline{C}_1
\left\|\sum\limits_{n=n_2}^{\infty}\dfrac{\alpha_\varepsilon\lambda_n^{b-1}\left<u(0),\phi_n\right>\phi_n}
{1+\alpha_\varepsilon\lambda_n^b}\right\|\notag\\
&= \overline{C}_1
\left\|\sum\limits_{n=n_2}^{\infty}\lambda_n^{-1}\dfrac{\alpha_\varepsilon\lambda_n^{b}}{1+\alpha_\varepsilon\lambda_n^{b}}
\left<u(0),\phi_n\right>\phi_n
\right\|\notag\\
&\geq\dfrac{\overline{C}_1}{2}
\left\|\sum\limits_{n=n_2}^{\infty}\lambda_n^{-1}\left<u(0),\phi_n\right>\phi_n\right\|.
\end{align}
 From \eqref{sbl1}--\eqref{ac2}, there exists a
constant $\overline{C}_{16}>0$ such that
\begin{align}\label{sbl2}\|u(0)-w_{\alpha_\varepsilon}(0)\|^2\leq
\overline{C}_{16}\| u(0)\|_p^{\frac{2}{p+1}}\varepsilon^{\frac{2p}{p+1}}.
\end{align}
Hence, there exists a constant $\overline{C}_{17}>0$ such
that
\begin{align}
\|u(0)-v_{\alpha_\varepsilon}(0)\| & \leq \|u(0)-w_{\alpha_\varepsilon}(0)\|
+\|v_{\alpha_\varepsilon}(0)-w_{\alpha_\varepsilon}(0)\| \nonumber \\
& \leq
\overline{C}_{17}\left(\| u(0)\|_p^{\frac{1}{p+1}}\varepsilon^{\frac{p}{p+1}}+\alpha_\varepsilon^{-1/b}\varepsilon+\varepsilon\right).\label{ac3}
\end{align}

It follows from Lemma~\ref{lemmamoi} that
\begin{align}
\tau\varepsilon&=\|B_{\alpha_\varepsilon}f-f\| =
\left\|\sum\limits_{n=1}^\infty\dfrac{\alpha_\varepsilon\lambda_n^b\left<f,\phi_n\right>\phi_n}
{1+\alpha_\varepsilon\lambda_n^b}\right\|\notag\\
&=
\left\|\sum\limits_{n=1}^\infty\dfrac{\alpha_\varepsilon\lambda_n^b\left<u(T),\phi_n\right>\phi_n}
{1+\alpha_\varepsilon\lambda_n^b}-
\sum\limits_{n=1}^\infty\dfrac{\alpha_\varepsilon\lambda_n^b\left<u(T)-f,\phi_n\right>\phi_n}
{1+\alpha_\varepsilon\lambda_n^b}\right\|\notag\\
&\le
 \left\|\sum\limits_{n=1}^\infty\dfrac{\alpha_\varepsilon\lambda_n^bE_{\gamma,1}
 \left(-\lambda_n T^\gamma\right)\left<u(0),\phi_n\right>\phi_n}
{1+\alpha_\varepsilon\lambda_n^b}\right\| +
\left\|\sum\limits_{n=1}^\infty\left<u(T)-f,\phi_n\right>\phi_n\right\|\notag\\
&\le
\overline{C}_2\left\|\sum\limits_{n=1}^\infty\dfrac{\alpha_\varepsilon\lambda_n^{b-1}\left<u(0),\phi_n\right>\phi_n}
{1+\alpha_\varepsilon\lambda_n^b}\right\| +
\varepsilon.\label{Theorem3estimate2}
\end{align}

If $0<p< b-1$,  using Lemma~\ref{Young's inequality}, we get
\begin{align}\label{z4}
\alpha_\varepsilon\lambda_n^b+1& \geq
\frac{b-p-1}{b}\left((\alpha_\varepsilon\lambda_n^b)^{\frac{b-p-1}{b}}\right)^{\frac{b}{b-p-1}}
+\frac{p+1}{b}.1^{\frac{b}{p+1}}\notag\\
&\geq (\alpha_\varepsilon\lambda_n^b)^{\frac{b-p-1}{b}}.
\end{align}
From (\ref{Theorem3estimate2}) and (\ref{z4}),  we have
\begin{align}\label{z5}
(\tau-1)\varepsilon&\le
\overline{C}_2\alpha_\varepsilon^{\frac{p+1}{b}}
\left\|\sum\limits_{n=1}^\infty
\lambda_n^{p}\left<u(0),\phi_n\right>\phi_n\right\|\leq\overline{C}_2\alpha_\varepsilon^{\frac{p+1}{b}}\| u(0)\|_p.
\end{align}

Hence, from  \eqref{ac3}  and \eqref{z5} and
$\varepsilon\leq\varepsilon^{\frac{p}{p+1}}\| u(0)\|_p^{\frac{1}{p+1}}$, we
arrive at the conclusion of part (i) of Theorem \ref{Theorem3}.

{\bf Proof of part (ii) of Theorem \ref{Theorem3}}.

If $p\geq b-1$, then from (\ref{Theorem3estimate2}) we have
\begin{align}\label{z5a}
(\tau-1)\varepsilon&\le \overline{C}_2\alpha_\varepsilon
\left\|\sum\limits_{n=1}^\infty
\lambda_n^{b-1}\left<u(0),\phi_n\right>\phi_n\right\|\notag\\
&\leq\overline{C}_2\lambda_1^{b-1-p}\alpha_\varepsilon
\left\|\sum\limits_{n=1}^\infty
\lambda_n^{p}\left<u(0),\phi_n\right>\phi_n\right\|\notag\\
&\leq\overline{C}_2\lambda_1^{b-1-p}\alpha_\varepsilon \| u(0)\|_p.
\end{align}
The conclusion of Part (ii) of Theorem \ref{Theorem3} is followed from  \eqref{ac3}  and \eqref{z5a}. The proof is complete.

\section{Numerical implementation and examples}\label{sec:num}

In this section we discuss the numerical implementation of the proposed regularization method for problem (\ref{main-problem}) and present some numerical tests for one and two dimensional equations. To focus our discussion on the performance of the regularization method, we chose the operator $A$ in such a way that its eigenvalues and eigenfunctions are explicitly available. This choice avoids possible misleading results  due to error in the calculation of the eigenvalues and eigenfunctions.

In our numerical implementation, given the eigenvalues and eigenfunctions of operator $A$,  the data $f = u(T)$ was generated by solving the forward problem (\ref{main-problem000}) using expansion (\ref{congthucnghiem}). The Mittag-Leffler functions $E_{\gamma,
1}\left(-\lambda_nt^\gamma\right)$ were computed using an implementation in Matlab by Roberto Garrappa which is available for download at

https://www.mathworks.com/matlabcentral/fileexchange/48154-the-mittag-leffler-function.
 We approximated the infinite series in  (\ref{congthucnghiem}) by the following sum
\begin{equation}\label{eq51}
 u(t) \approx \sum\limits_{n=1}^{N_p}
E_{\gamma,
1}\left(-\lambda_nt^\gamma\right)\left<u_0,\phi_n\right>\phi_n, \quad t > 0.
\end{equation}
To simulate noisy data, we added an additive uniformly distributed random noise of $L^2$-norm $\varepsilon$  to $u(T)$ to obtain data $f$. Given $f$  and the  parameters $E$, $b$, $p$, $\epsilon$, the  algorithm for calculating $u(t), 0 \le t < T,$ includes two steps:
\begin{itemize}
 \item Step 1: Calculate the regularization parameter $\alpha$ using either the \textit{a priori} or \textit{a posteriori} choice rules described in Theorems \ref{Theorem2} and \ref{Theorem3}. In the \textit{a posteriori} choice rule, $\alpha$ is found as the unique solution of (\ref{a-posteriori}). Here, we approximate operator $B_{\alpha}f$ by a finite sum.
 \item Step 2: Calculate the regularized solution $v_\alpha$ using the following approximation of the explicit formula (\ref{regularization-solution})
\begin{equation}\label{eq52}
 v_{\alpha,\varepsilon}(t) :=\sum\limits_{n=1}^{N_{i}} \dfrac{E_{\gamma,1}\left(-\lambda_{\alpha
n} t^\gamma\right) \left<f,\phi_n\right>\phi_n} {
E_{\gamma,1}\left(-\lambda_{\alpha n} T^\gamma\right)},\quad \forall
t\in[0,T].
\end{equation}
\end{itemize}

In general, the number of basis functions $N_i$ used in the inverse problem is not necessary equal to  the number of basis functions $N_p$ used in the approximation of the solution of the forward problem. In fact, we have observed through our numerical tests that $N_i$ may have to be chosen smaller than $N_p$ to avoid numerical instabilities in the solution of the backward equation. It was also mentioned in \cite{hdt} that $N_i$ is another regularization parameter that should be carefully chosen along with $\alpha$.


\noindent\textbf{Example 1}: In this example, we consider the one-dimensional problem
\begin{eqnarray}
  \frac{\partial^\gamma u(x,t)}{\partial t^\gamma} &=& \frac{\partial^2 u(x,t)}{\partial x^2},\quad x \in (0,\pi),\ t \in (0,T), \nonumber \\
 u(0,t) &=& u(\pi,t)=0, \quad t\in (0,T),\label{eq53}\\
u(x,0) &=& u_0(x):= \sin(x) + \sin(2x) + \sin(3x),\quad x\in [0,\pi].\nonumber
\end{eqnarray}

For this example, the eigenvalues $\lambda_n = n^2$ and orthonormal eigenfunctions $\phi_n(x) = \sqrt{\frac{2}{\pi}}\sin(nx)$, $n = 1, 2,\dots$. Moreover,
$$\left<u_0,\phi_n\right> =
    \begin{cases}
        \sqrt{\frac{\pi}{2}} = \dfrac{1}{\|\phi_n\|}, & n \leq 2 \\
         0, & n > 3.
    \end{cases}
$$
Therefore, the solution $u(x,t)$ of (\ref{zz1}) is given by
$$ u(x,t) = E_{\gamma,1}(- t^\gamma ) \sin(x) + E_{\gamma,1}(-4 t^\gamma ) \sin(2x)  + E_{\gamma,1}(-9 t^\gamma ) \sin(3x).$$
That means, (\ref{eq51}) becomes a true equality for $N_p = 3$. In the backward problem, if the data is exact, only 3 terms in (\ref{eq52}) are needed for calculating $v_\alpha$ exactly. However, since we expect the data to be noisy, we chose $N_i = 5$. This choice seems to be optimal for all the one-dimensional examples we discuss in this paper. This choice was also considered in \cite{hdt} in similar tests.

First, we analyzed the effect of parameter $p$ on the performance of the proposed algorithm. To this end, we chose $\gamma = 1/2$, $T = 1$, $b = 4$, and considered 6 noise levels of 0.1\%, 0.2\%, 0.4\%, 0.8\%, 1.6\%, and 3.2\%.   The relative $L^2$-norm error at time $t$ is defined by
$$e_r(\epsilon,t) := \frac{\| u(\cdot,t) - v_{\alpha,\epsilon} (\cdot,t)\|}{\|u(\cdot,t)\|}\times 100(\%).$$

Table \ref{tab1} shows the relative $L^2$-norm error at $t = 0$ for  three values of $p$: $p = 1$, $p=2$, and $p=3$. We can observe that the error generally decreases as the measured error decreases. Moreover, the larger $p$ is, the smaller the reconstruction error will be. This is consistent with the error estimates in Theorem \ref{Theorem2}.  Since the behavior is similar for the \textit{a posteriori} parameter choice rule, we do not present it here.

\begin{table}[tph]
 \centering
 \begin{tabular}{|c|c|c|c|c|c|c|}
  \hline Noise & 0.1\% & 0.2\% & 0.4\% & 0.8\% & 1.6\% & 3.2\%  \\
  \hline  $p = 1$ & 0.08  &  0.29 &   0.86 &   0.33 &   4.91 &   4.46   \\
 \hline  $p = 2$ & 0.07 &    0.28 &   0.75 &    0.39&    2.30 &    4.47\\
  \hline  $p = 3$ &0.06 &   0.22 &   0.42 &   0.89 &   1.72 &   4.07 \\
   \hline
 \end{tabular}
 \caption{Example 1: Relative $L^2$-norm error $e_r(\varepsilon,0)$ (\%) at $t=0$ for $\gamma = 1/2$, $b = 4$ and $p = 1, 2, 3$. Measured data was given at $T = 1$. The regularization parameter $\alpha$ was chosen using the \textit{a priori} rule. For small error levels, the error decreases with respect to $p$, which is consistent with the error estimates in Theorem \ref{Theorem2}.}
\label{tab1}
\end{table}

In Figure \ref{fig:11} we compared the reconstruction results with the exact initial solution $u(x,t)$  at $t = 0.1$ and $t = 0$ for $p = 3$ and $b = 4$ at two noise levels: $2\%$ and $5\%$. The figure shows quite accurate reconstructions of $u(x,t)$ at $t = 0$ for both \text{a priori} and \textit{a posteriori} parameter choice rules. In the latter case, the parameter $\tau$ in (\ref{a-posteriori}) was chosen as $\tau = 1.05$.  The reconstructions at $t =0.1$ are more accurate, as expected because problem (\ref{main-problem}) is well-posed for $t > 0$. Qualitatively, the accuracy of our results is comparable to those presented in Figure 2 of \cite{hdt} .

Next, we considered the case with $\gamma = 3/4$. We also chose $p$ and $b$ as in the previous case. The results for noise levels of $2\%$ and $5\%$ are depicted in Figure \ref{fig:12}. We also obtained reasonably accurate results for both cases. However, the accuracy of the case $\gamma = 1/2$ is higher near $t = 0$ as shown in Figure \ref{fig:13} in which the $L^2$-norm error profile with respect to time is shown for noise level of $5\%$. Figure \ref{fig:13} also shows that the both parameter choice rules produced comparable results for $\gamma = 1/2$ but the \textit{a priori} parameter choice rules gives more accurate than the \textit{a posteriori} parameter choice rule for $\gamma = 3/4$.
\begin{figure}[tph]
 \centering
 \begin{tabular}{c c c}
  \includegraphics[width = 0.45\textwidth]{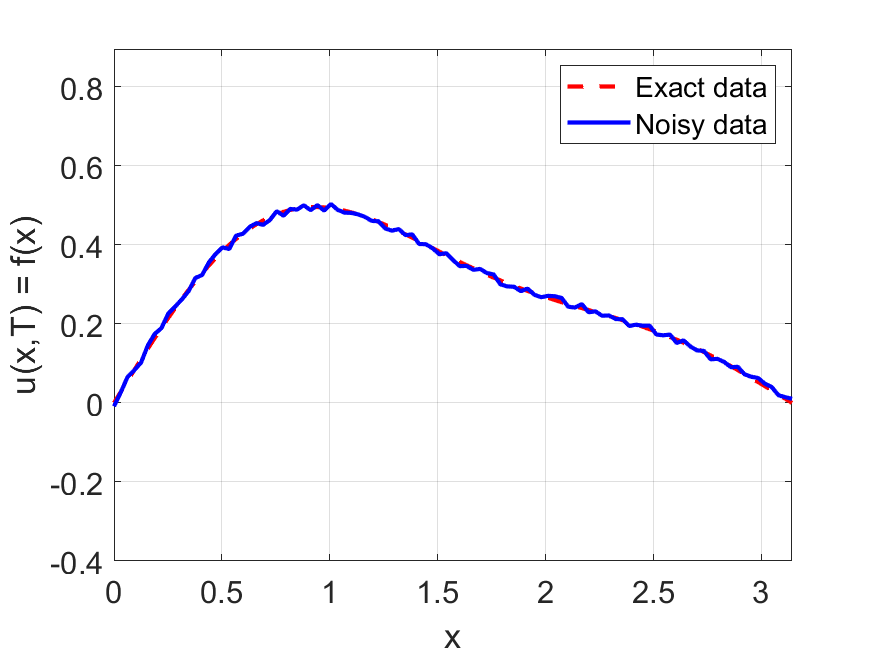} &
  \includegraphics[width = 0.45\textwidth]{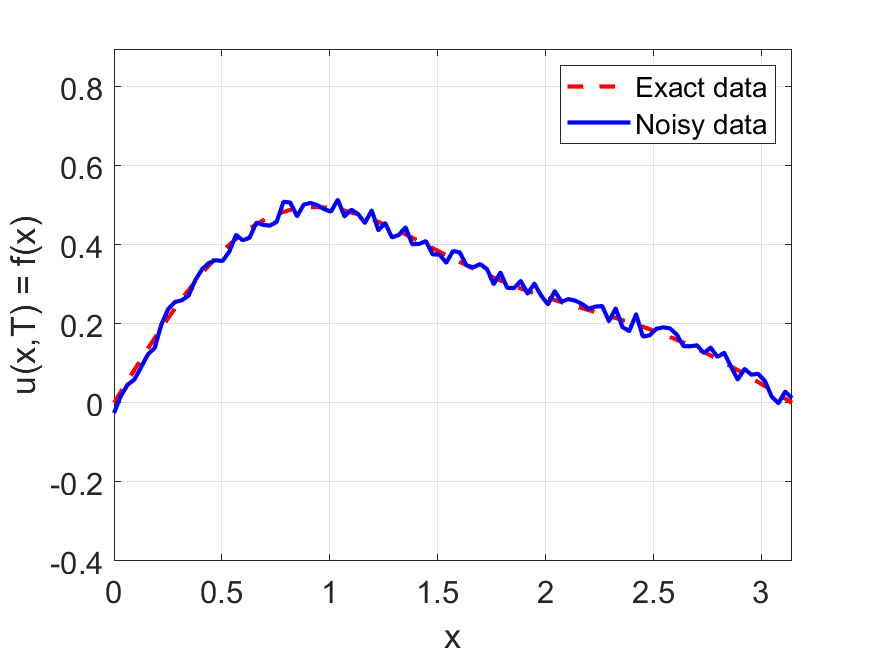} \\
     (a) Data, noise = 2\% & (b) Data, noise = 5\% \\
     & \\
 \includegraphics[width = 0.45\textwidth]{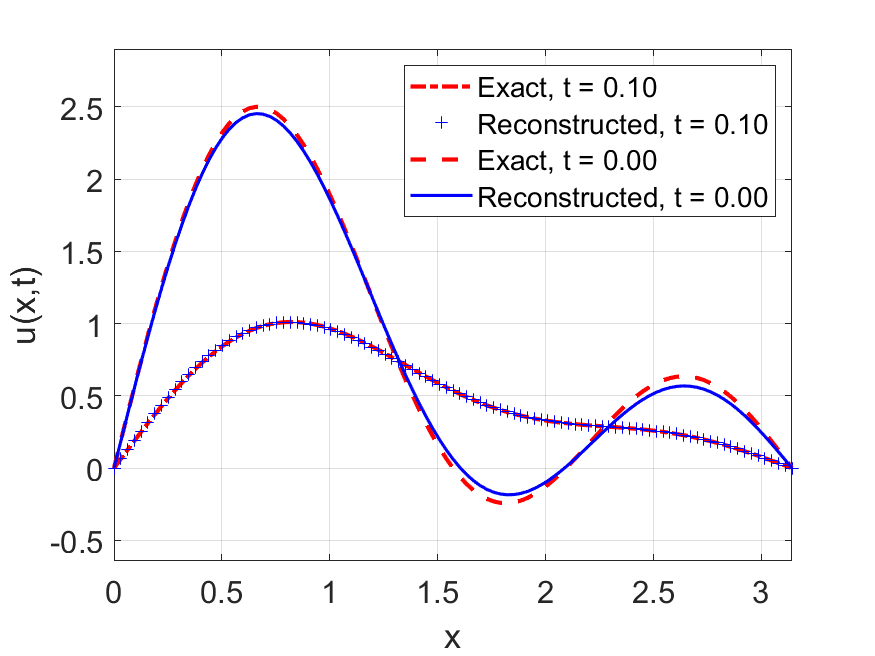} &
  \includegraphics[width = 0.45\textwidth]{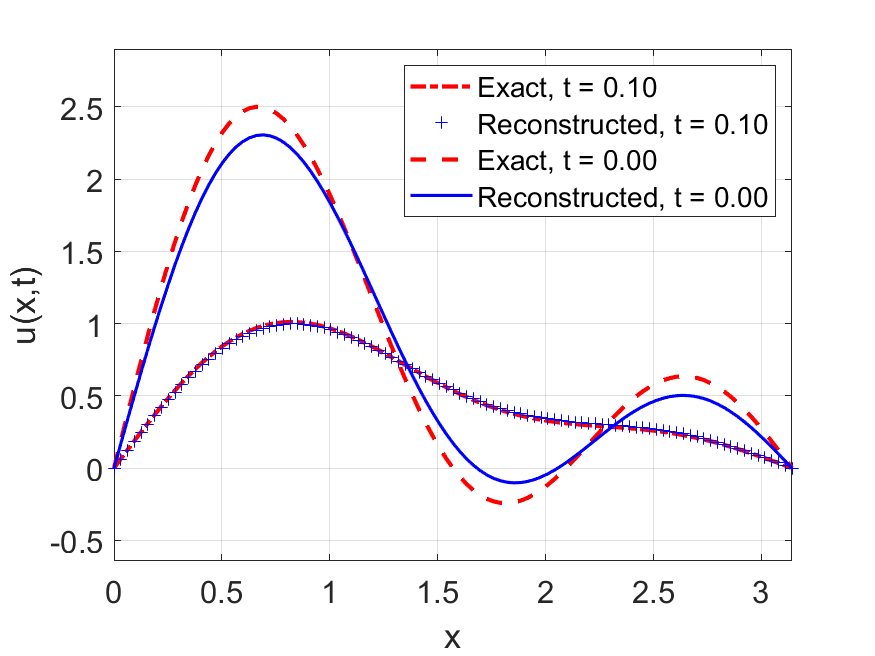} \\
     (c) \textit{a priori} choice, noise = 2\% & (d)   \textit{a priori} choice, noise = 5\%  \\
     & \\
   \includegraphics[width = 0.45\textwidth]{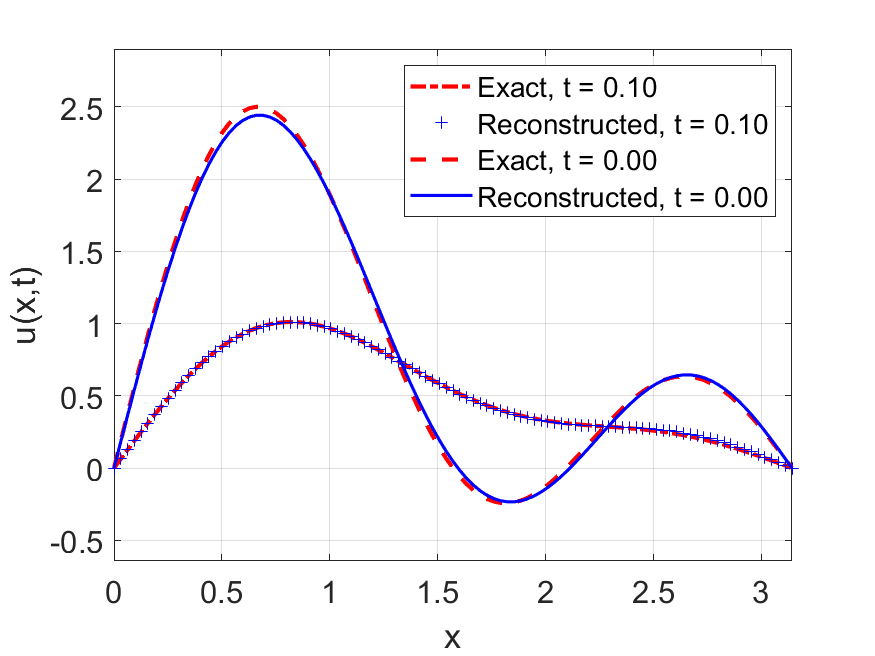} &
  \includegraphics[width = 0.45\textwidth]{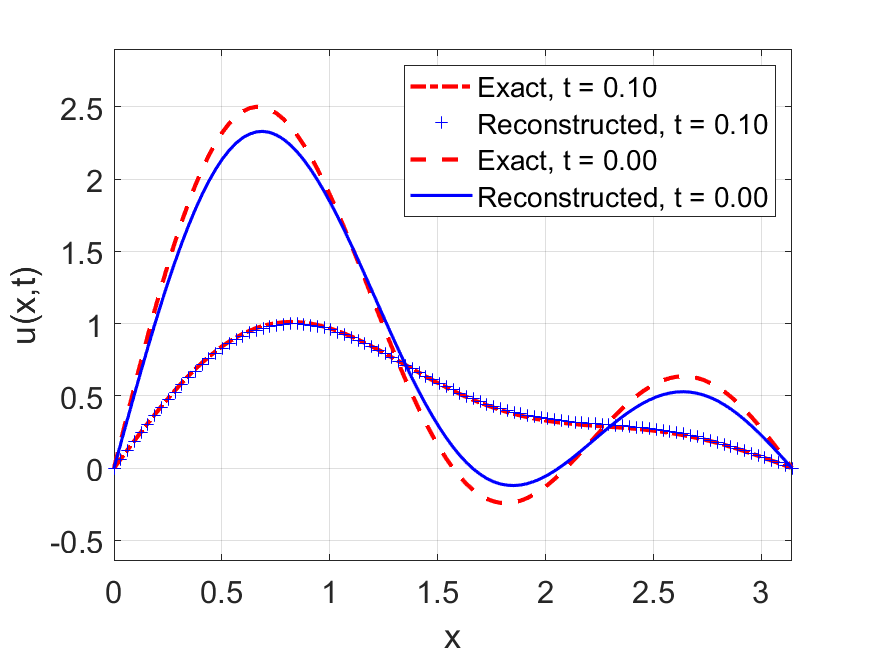} \\
     (e) \textit{a posteriori} choice, noise = 2\% & (f) \textit{a posteriori} choice, noise = 5\% \\

  \end{tabular}
 \caption{Reconstruction result for Example 1 for $\gamma = 1/2$, $p = 3$, $ b= 4$. Left column: noise = 2\%; right column: noise = 5\%.}
 \label{fig:11}
 \end{figure}

 \begin{figure}[tph]
 \centering
 \begin{tabular}{c c c}
  \includegraphics[width = 0.45\textwidth]{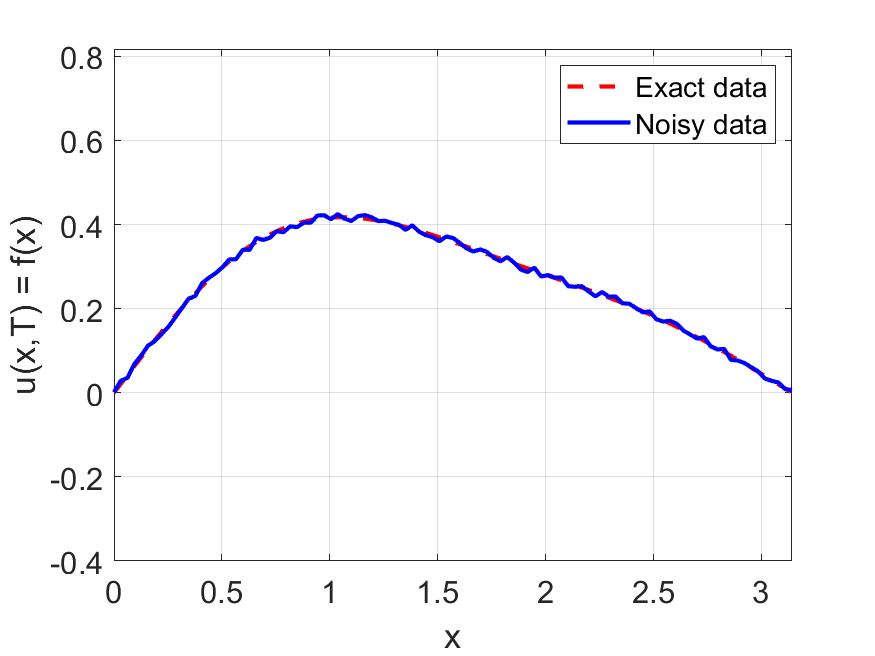} &
  \includegraphics[width = 0.45\textwidth]{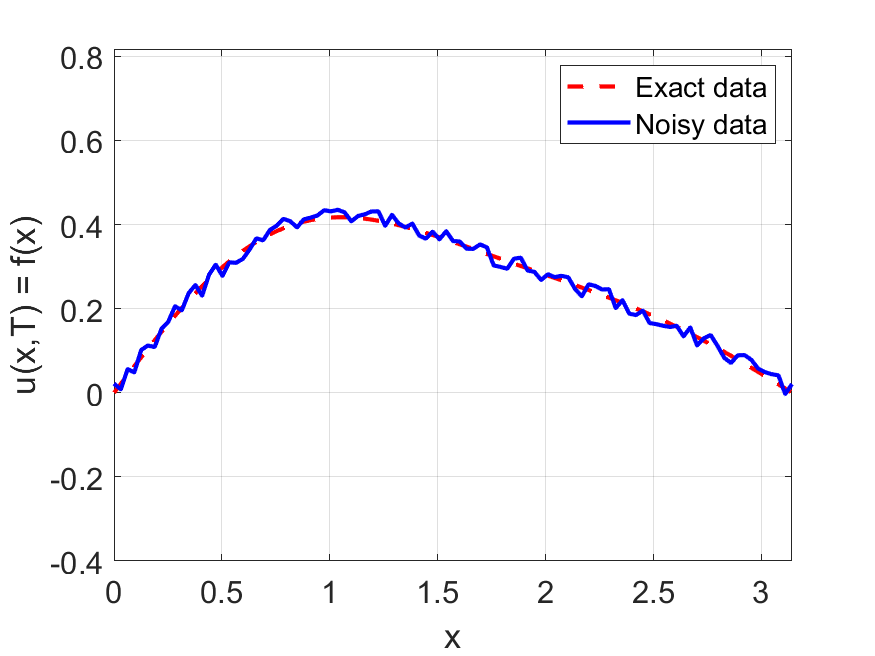} \\
     (a) Data, noise = 2\% & (b) Data, noise = 5\% \\
     & \\
 \includegraphics[width = 0.45\textwidth]{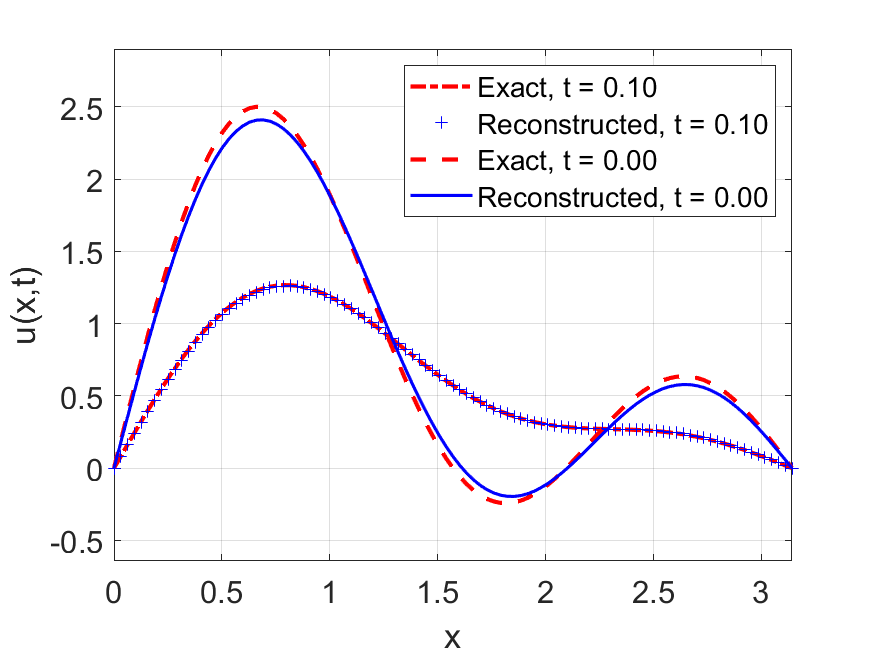} &
  \includegraphics[width = 0.45\textwidth]{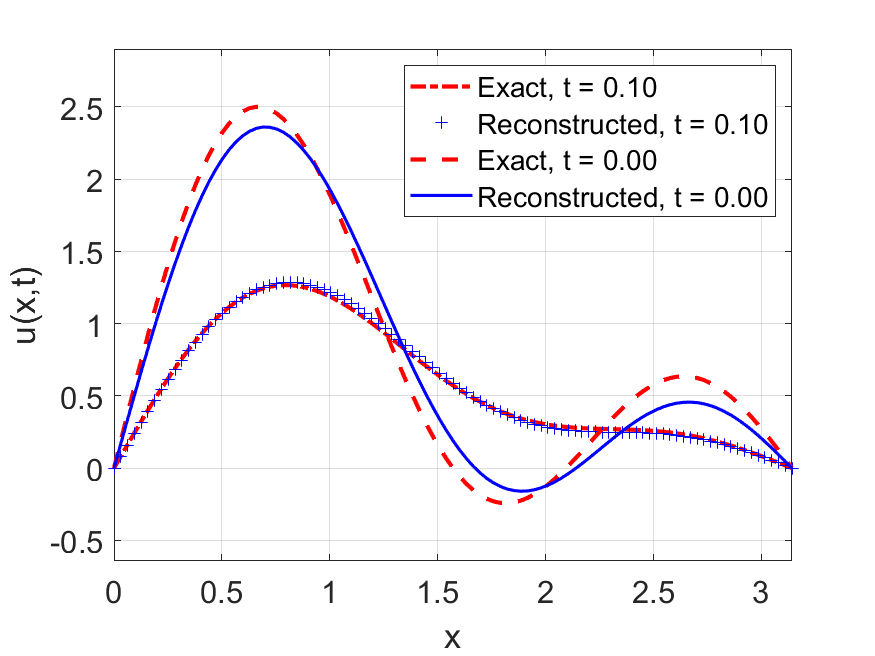} \\
     (c) \textit{a priori} choice, noise = 2\% & (d)   \textit{a priori} choice, noise = 5\%  \\
     & \\
   \includegraphics[width = 0.45\textwidth]{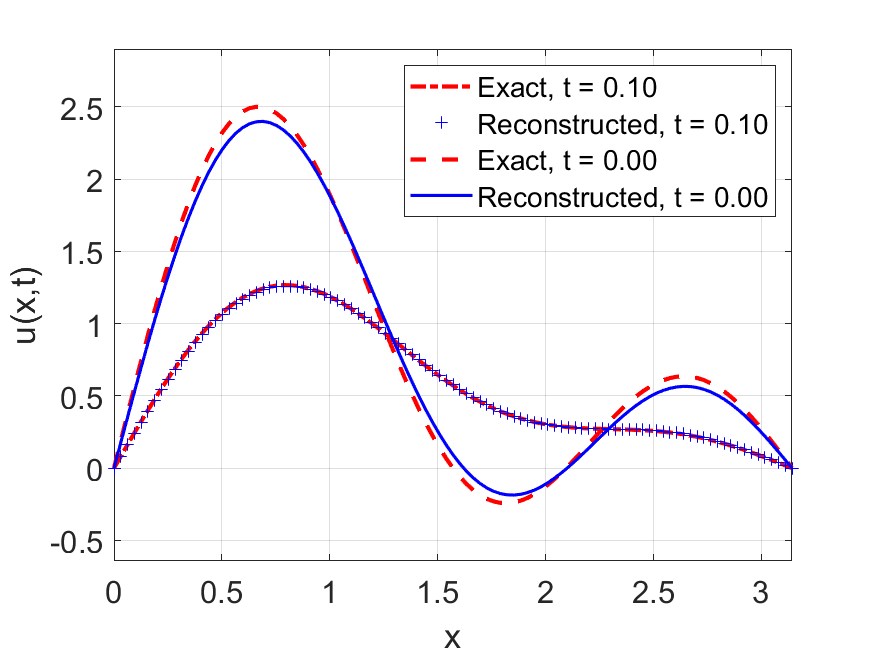} &
  \includegraphics[width = 0.45\textwidth]{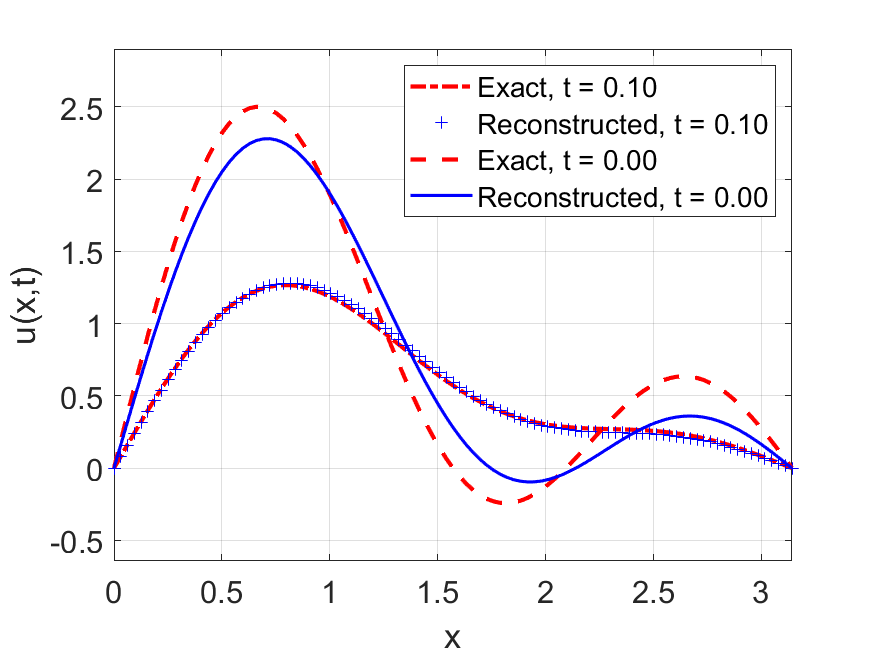} \\
     (e) \textit{a posteriori} choice, noise = 2\% & (f) \textit{a posteriori} choice, noise = 5\% \\

  \end{tabular}
 \caption{Reconstruction result for Example 1 for $\gamma = 3/4$, $p = 3$, $ b= 4$. Left column: noise = 2\%; right column: noise = 5\%.}
 \label{fig:12}
 \end{figure}

  \begin{figure}[tph]
 \centering
 \begin{tabular}{c c c}
  \includegraphics[width = 0.45\textwidth]{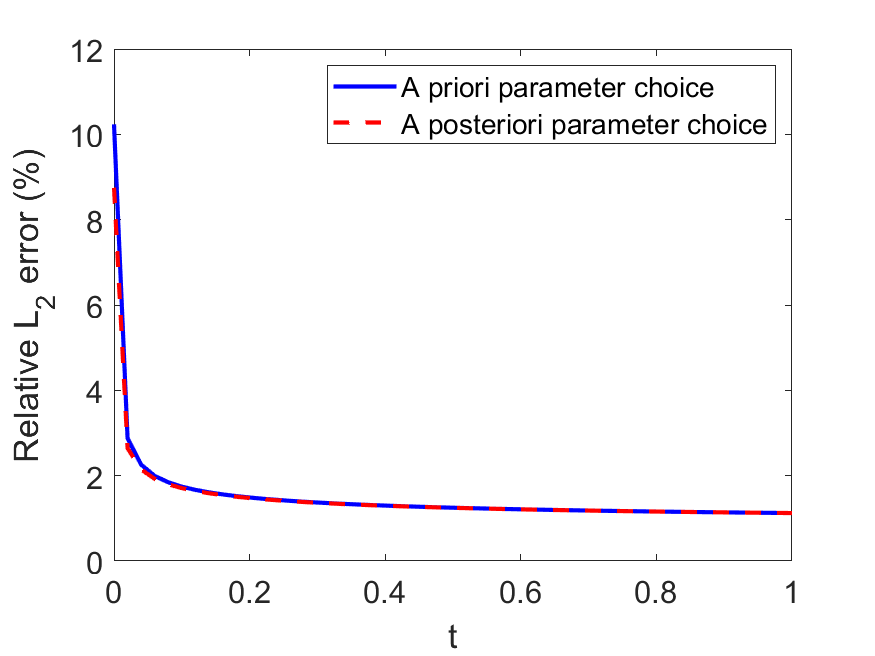} &
  \includegraphics[width = 0.45\textwidth]{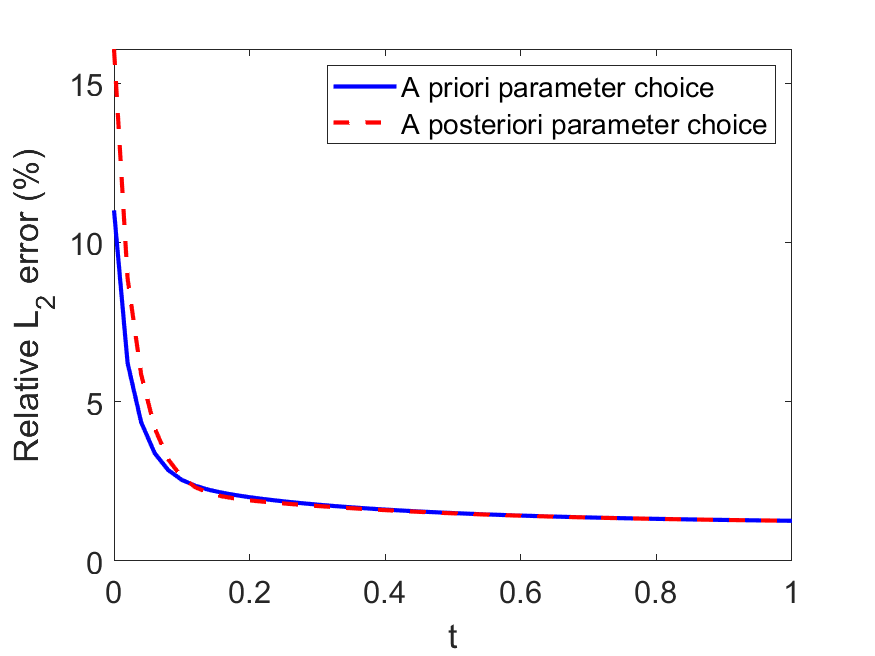} \\
     (a) $\gamma = 1/2$ & (b) $\gamma = 3/4$.
  \end{tabular}
 \caption{Distribution of relative $L^2$-norm error over time for Example 1 with noise level = 5\%. a) $\gamma = 1/2$; b) $\gamma = 3/4$. We can see that the error near $t = 0$ is larger for larger $\gamma$. That means, the larger $\gamma$, the more ill-posed the backward problem.}
 \label{fig:13}
 \end{figure}

  \noindent{\textbf{Example 2}: In this example we test the algorithm for another one dimensional problem which is described by the same equation as in (\ref{eq53}) but the initial condition is given by
  \begin{equation*}
   u(x,0) = u_0(x): = \begin{cases}
                        x, &  0\le x < \pi/2,\\
                        \pi - x, & \pi/2 \le x\le \pi.
                      \end{cases}
  \end{equation*}
It is easy to verify that $$\left< u_0,\phi_n\right> = \sqrt{\frac{2}{\pi}}\int_0^{\pi} u_0(x) \sin(nx) dx = \frac{2\sqrt{2}}{n^2\sqrt{\pi}}\sin(\frac{n\pi}{2}).$$ We approximated the solution of the forward problem (\ref{main-problem000}) by (\ref{eq51}) with $N_p = 30$. In solving the backward equation, we again chose $N_i = 5 $ as in Example 1. We also chose  $p = 3$ and $b = 4$.

The solution values of the backward problem at $t = 0.1$ and $t = 0$ are shown in Figure~\ref{fig:21} for noise levels of $2\%$ and $5\%$. The figure shows that the reconstruction looks very accurate for both parameter choice rules at $t = 0.1$ and reasonably good at $t = 0$. These results are also comparable to the results obtained in Figure 6 of \cite{hdt}. Note that due to the diffusion process, the nonsmooth initial condition is smoothed out rapidly in time, making the reconstruction of the nonsmooth behavior really challenging. We also observe that the \textit{a priori} parameter choice gave more accurate reconstructions of the initial condition $u_0$, in particular, near the point $x = \pi/2$ at which the initial condition is not smooth. One possible reason for the this due to the approximation of operator $B_\alpha$ in calculating the regularization parameter $\alpha$ in the \textit{a posteriori} choice, which may not result in the optimal value of $\alpha$.

 \begin{figure}[tph]
 \centering
 \begin{tabular}{c c c}
  \includegraphics[width = 0.45\textwidth]{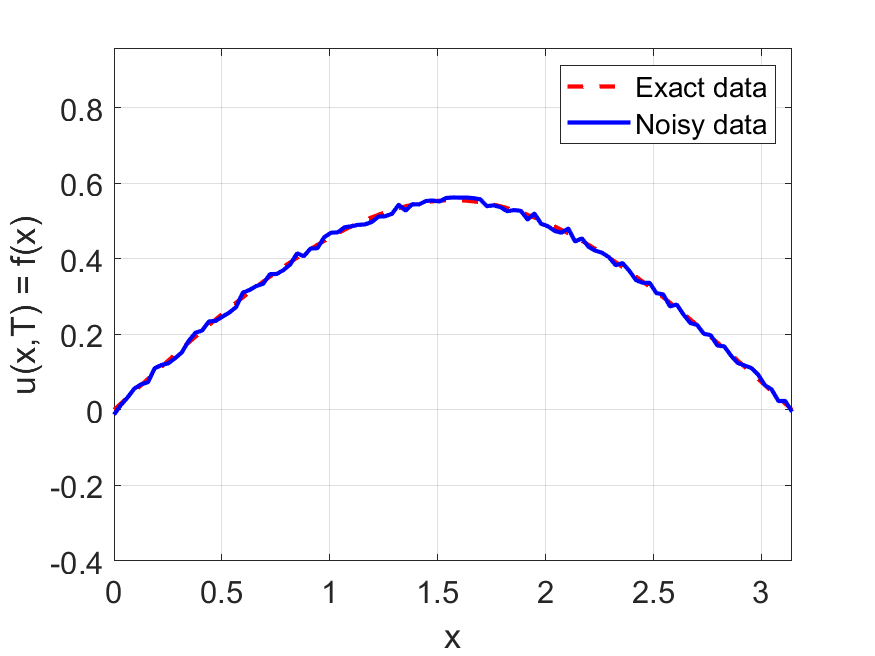} &
  \includegraphics[width = 0.45\textwidth]{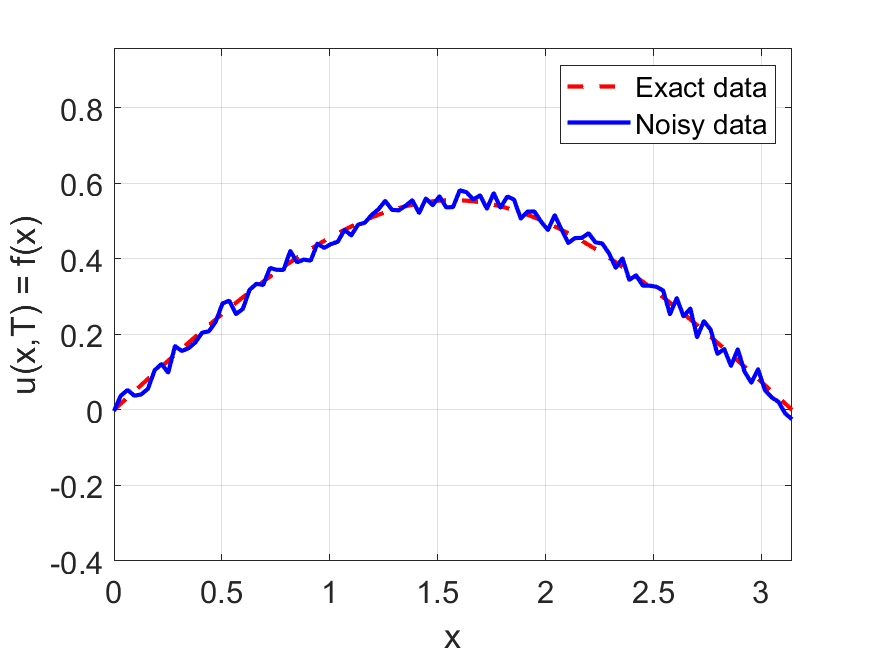} \\
     (a) Data, noise = 2\% & (b) Data, noise = 5\% \\
     & \\
 \includegraphics[width = 0.45\textwidth]{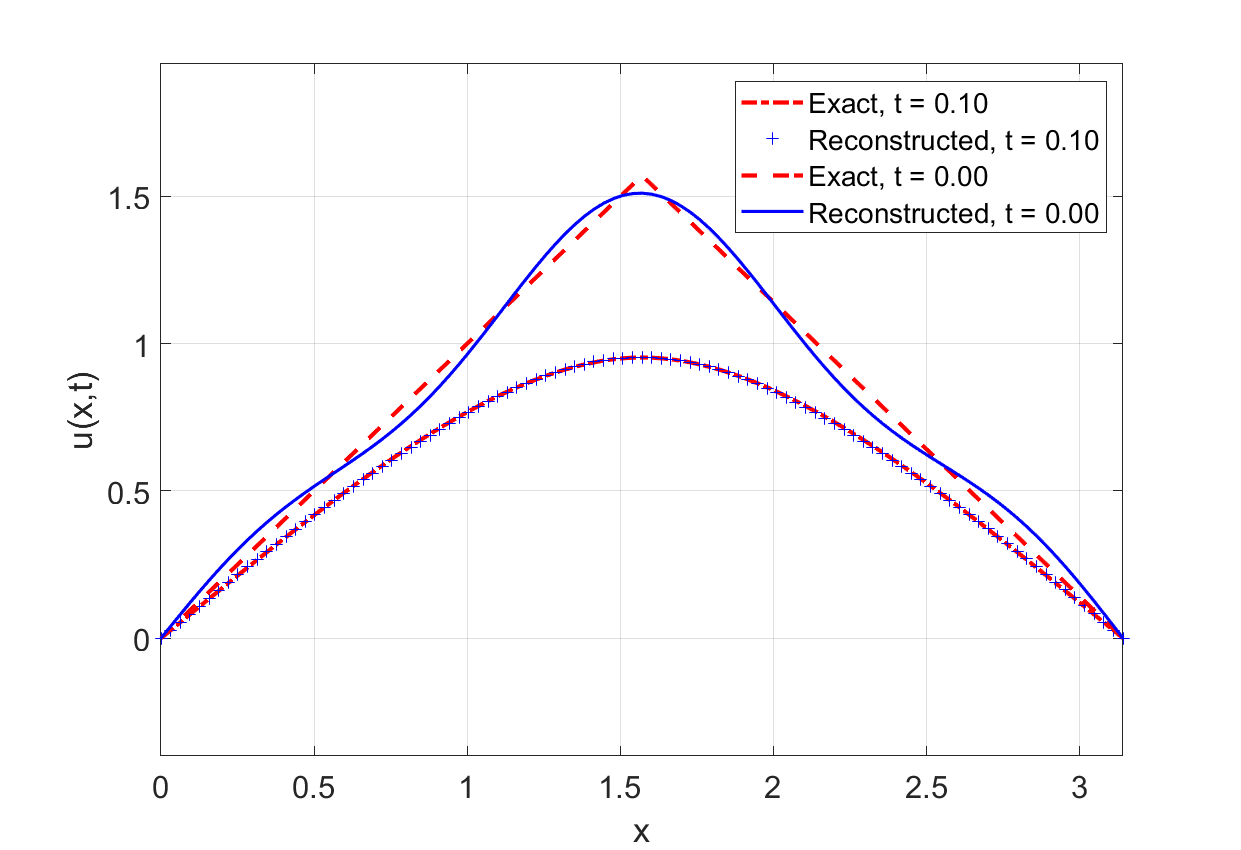} &
  \includegraphics[width = 0.45\textwidth]{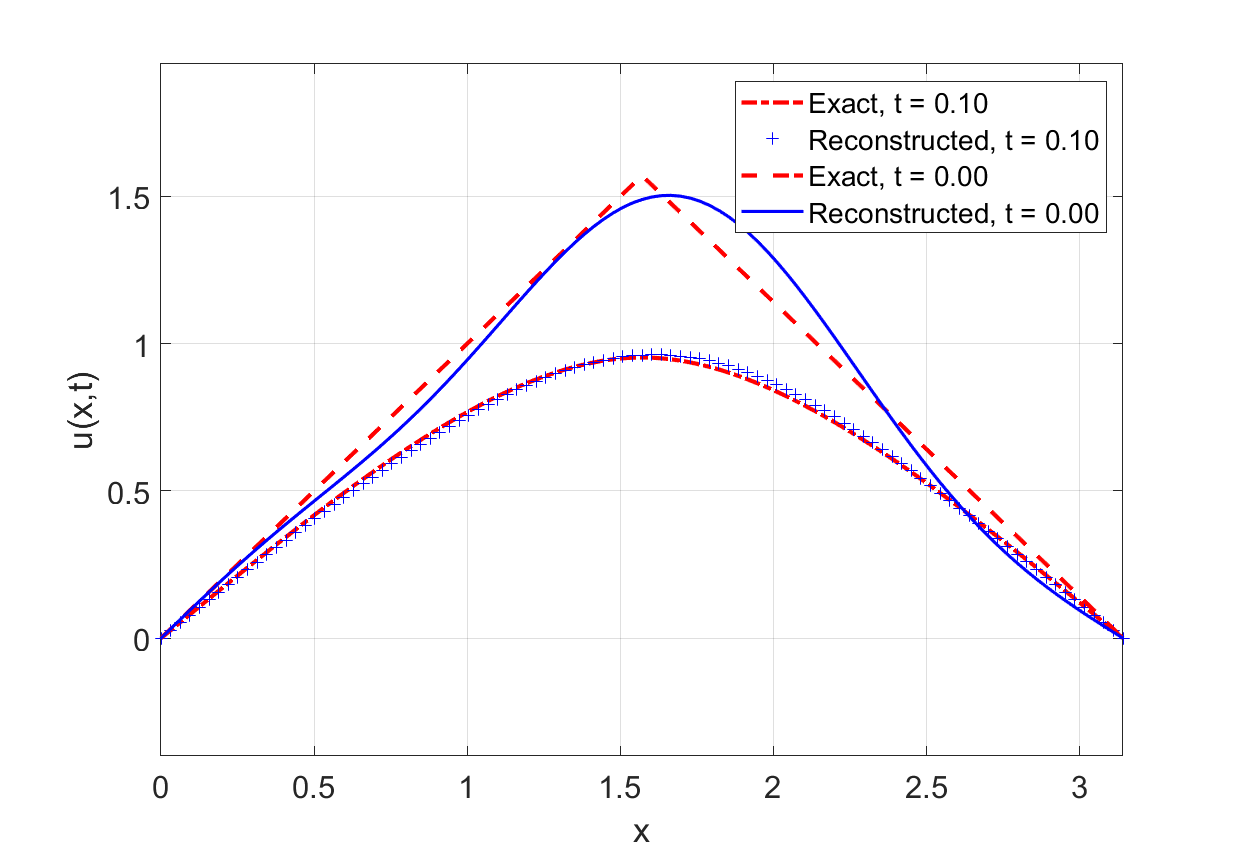} \\
     (c) \textit{a priori} choice, noise = 2\% & (d)   \textit{a priori} choice, noise = 5\%  \\
     & \\
   \includegraphics[width = 0.45\textwidth]{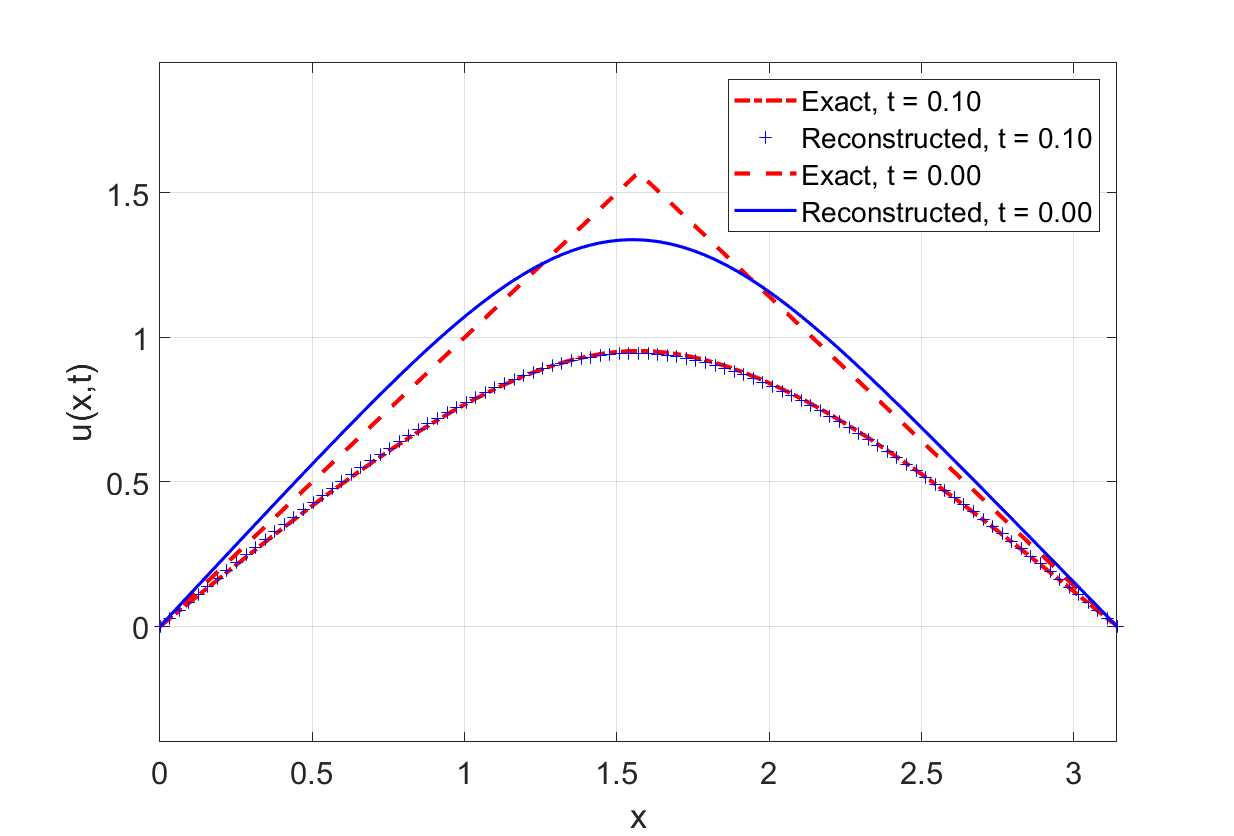} &
  \includegraphics[width = 0.45\textwidth]{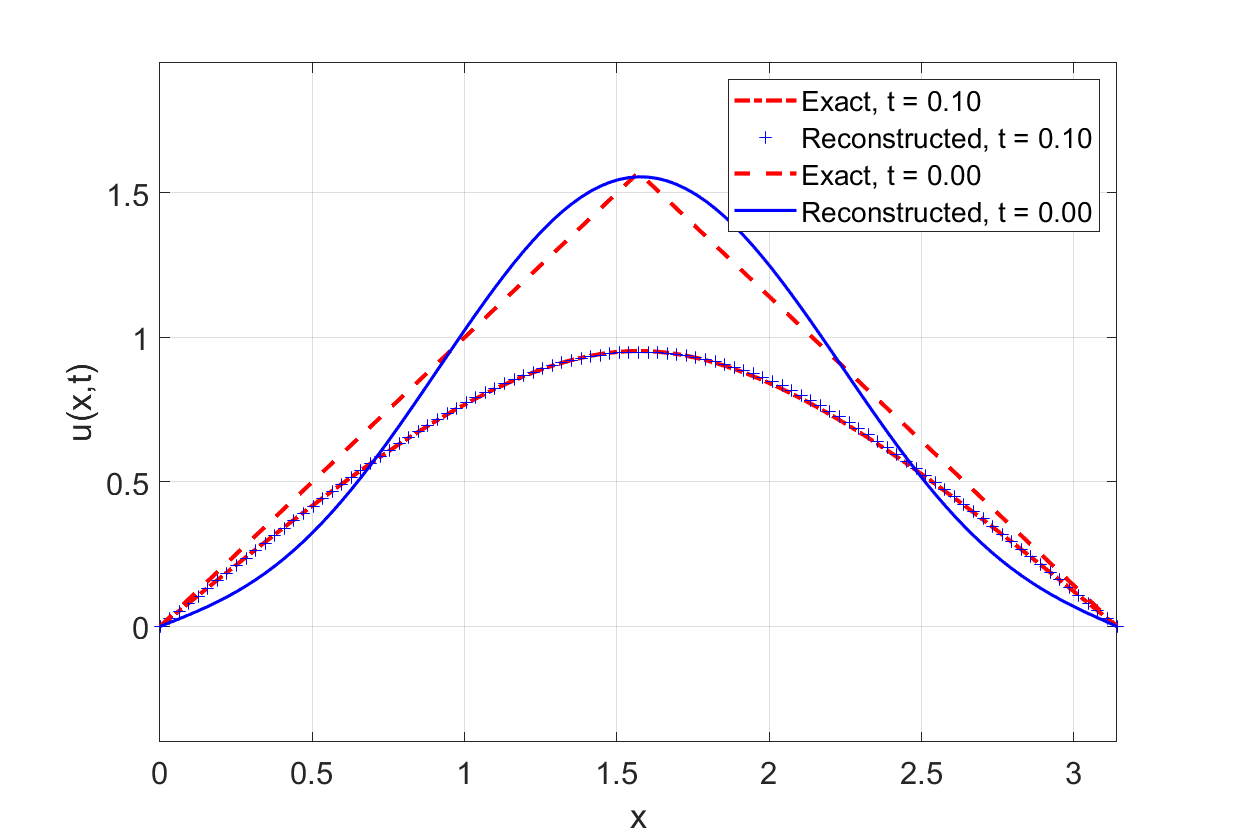} \\
     (e) \textit{a posteriori} choice, noise = 2\% & (f) \textit{a posteriori} choice, noise = 5\% \\

  \end{tabular}
 \caption{Reconstruction result for Example 2 for $\gamma = 1/2$, $p = 3$, $ b= 4$. Left column: noise = 2\%; right column: noise = 5\%.}
 \label{fig:21}
 \end{figure}

 \noindent{\textbf{Example 3}: As the last numerical example, we tested the algorithm against the following two-dimensional problem
 \begin{eqnarray}
  \frac{\partial^\gamma u(x,y,t)}{\partial t^\gamma} &=& \frac{\partial^2 u(x,y,t)}{\partial x^2} + \frac{\partial^2 u(x,y,t)}{\partial y^2},\quad x \in (0,\pi),\ y\in (0,\pi), \ t \in (0,T), \nonumber \\
 u(0,y,t) &=& u(\pi,y,t)=0, \quad y\in (0,\pi),\ t\in (0,T),\label{eq54}\\
 u(x,0,t) &=& u(x,\pi,t)=0, \quad x\in (0,\pi),\ t\in (0,T),\nonumber \\
u(x,y,0) &=& u_0(x,y):= \sin(x)\sin(y)  + \sin(2x) \sin(y),\quad x\in [0,\pi], \ y \in [0,\pi].\nonumber
\end{eqnarray}

 For this problem, the eigenvalues and eigenfunctions are given by
 $$ \lambda_{nm} = n^2 + m^2,\quad \phi_{nm}(x,y) = \frac{2\sin(nx)\sin(my)}{\pi},\quad n, m = 1,2,\dots$$
 The inner products $\left<u_{0}, \phi_{nm}\right>$ are given by
 \begin{equation*}
  \left<u_{0}, \phi_{nm}\right> = \begin{cases}
                                   \pi/2, & n = 1, 2 \text { and } m = 1,\\
                                   0, & \text{ otherwise}.
                                  \end{cases}
 \end{equation*}

The implementation of the algorithm for this problem was similar to the one-dimensional case, except that we had to flatten the matrices of the eigenvalues and eigenfunctions to obtain one-dimensional arrays and then sorted them in the nondecreasing order. Note that in this example, there are repeated eigenvalues, but the corresponding eigenfunctions are not the same.

The reconstructions of the initial condition are shown in Figures \ref{fig:31}-\ref{fig:32} for two noise levels of $2\%$ and $10\%$ with $p = 3$, $b =4$ and $\gamma = 1/2$.  In this example, we observed that $N_i = 10 $ is a good truncation number in solving the backward problem.  The figures indicate that the initial condition was reconstructed accurately taking into account the noise levels in the measured data.

  \begin{figure}[tph]
 \centering
 \begin{tabular}{c c c}
  \includegraphics[width = 0.45\textwidth]{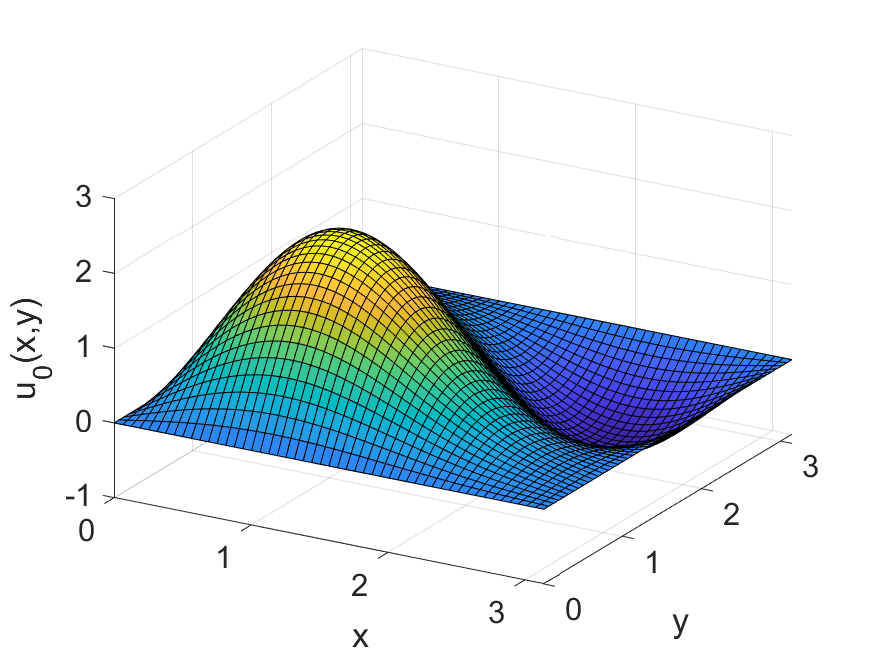} &
  \includegraphics[width = 0.45\textwidth]{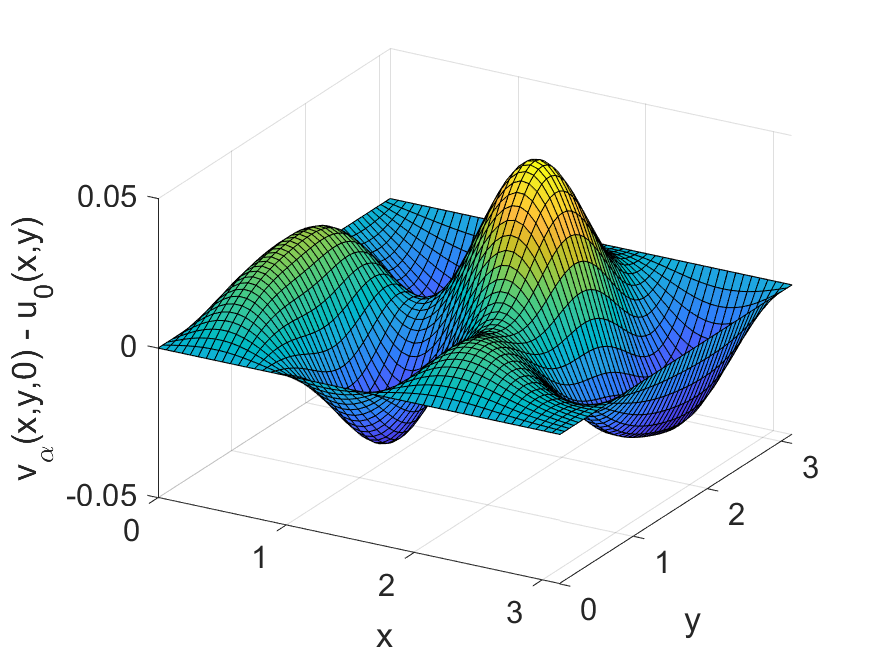} \\
     (a) Exact initial condition & (b) Error, \textit{a priori} rule\\
     & \\
 \includegraphics[width = 0.45\textwidth]{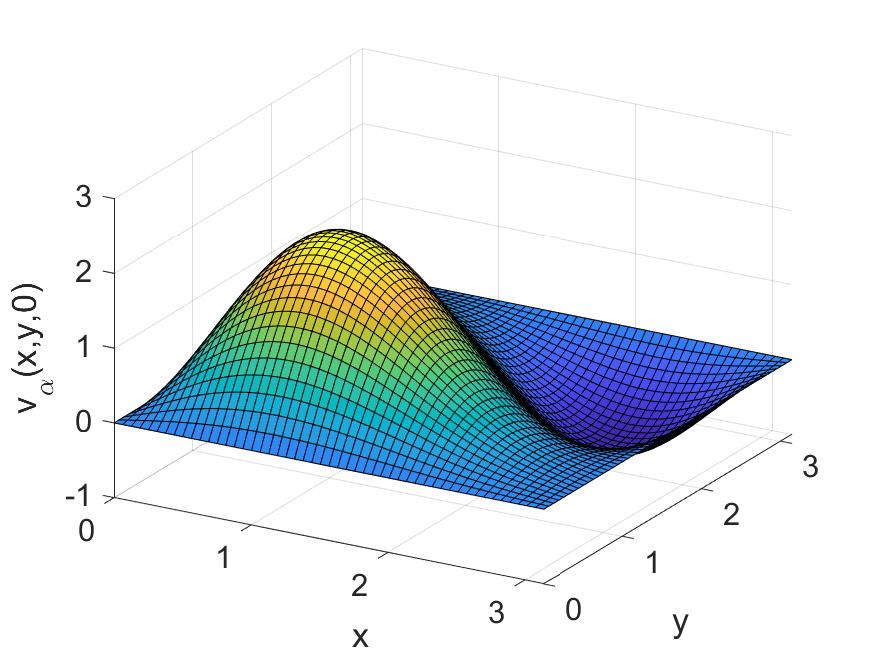} &
  \includegraphics[width = 0.45\textwidth]{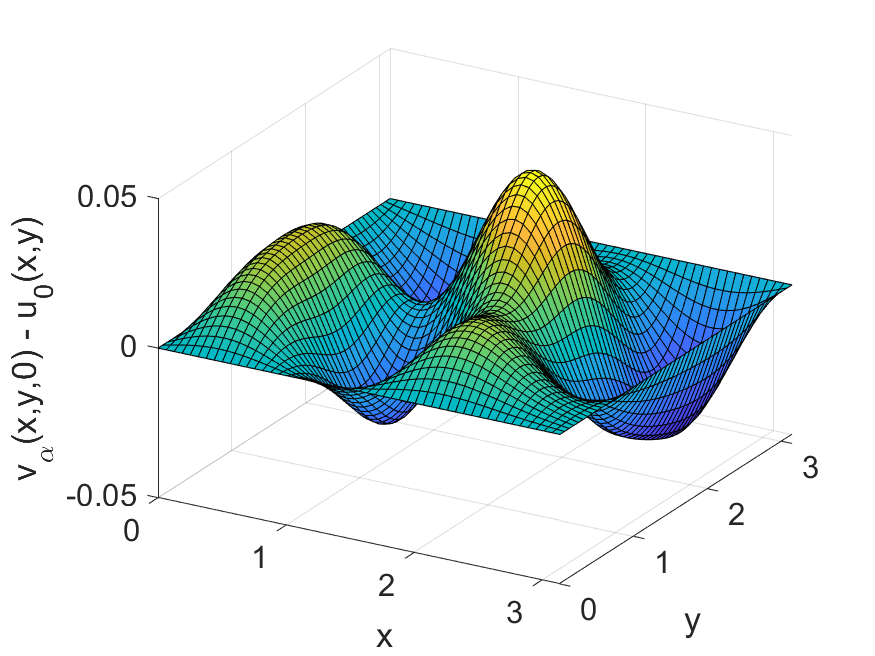} \\
     (c) Reconstruction, \textit{a priori} rule & (d)  Error, \textit{a posteriori} rule \\
         & \\
 \includegraphics[width = 0.45\textwidth]{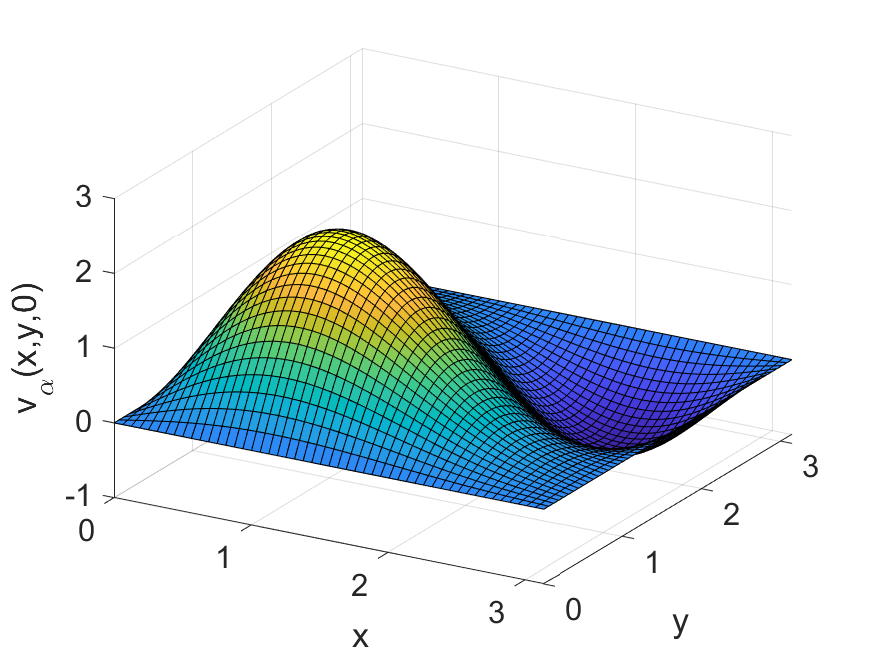} &
  \includegraphics[width = 0.45\textwidth]{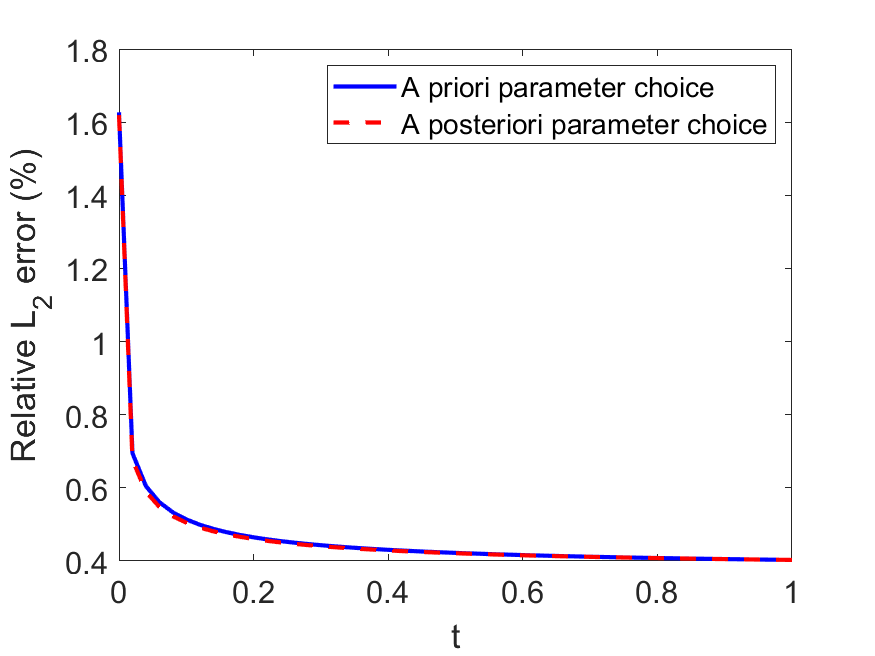} \\
     (e) Reconstruction, \textit{a posteriori} rule & (f)  $L^2$ error distribution in time \\
  \end{tabular}
 \caption{Reconstruction of the initial condition in Example 3  for $\gamma = 1/2$, $p = 3$, $ b= 4$ for noise level = 2\% using the \textit{a priori} parameter choice rule (c) and \textit{a posteriori} parameter choice rule (e). The maximum error is about 2.5\% and the $L^2$-norm error is about 1.6\%. Both parameter choice rules produced almost the same results. }
 \label{fig:31}
 \end{figure}

   \begin{figure}[tph]
 \centering
 \begin{tabular}{c c c}
  \includegraphics[width = 0.45\textwidth]{Exa4_exact_initcond.png} &
  \includegraphics[width = 0.45\textwidth]{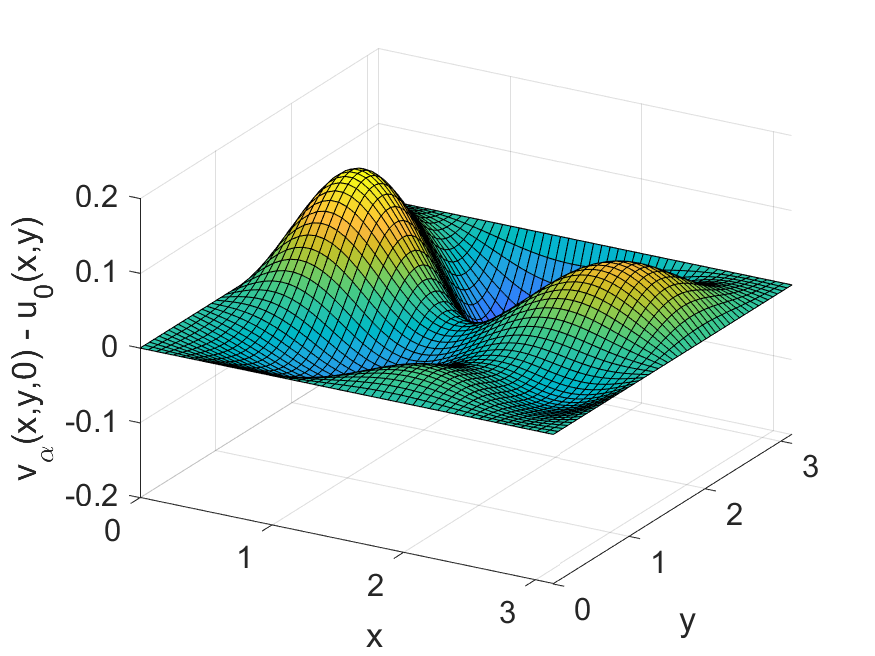} \\
     (a) Exact initial condition & (b) Error, \textit{a priori} rule\\
     & \\
 \includegraphics[width = 0.45\textwidth]{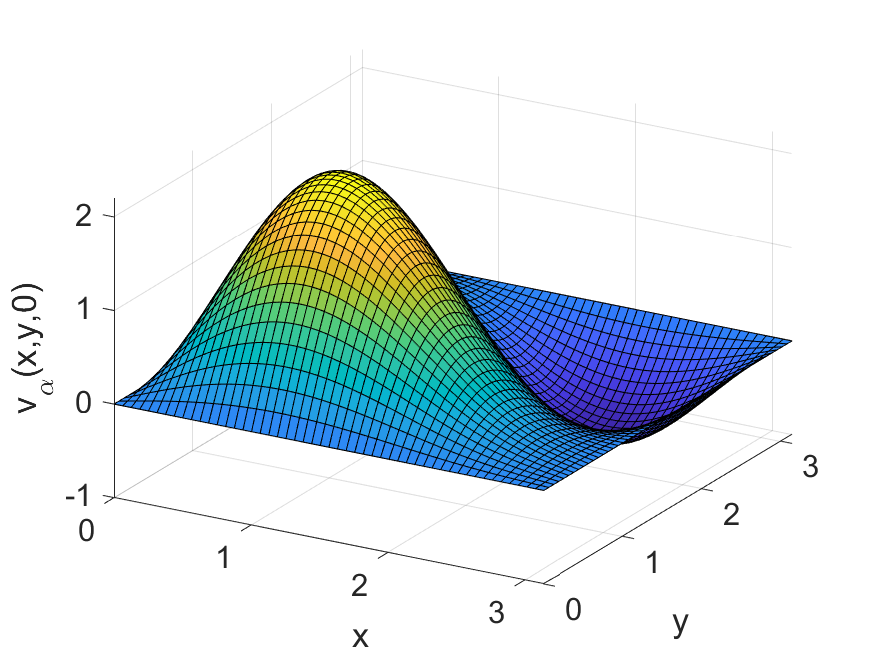} &
  \includegraphics[width = 0.45\textwidth]{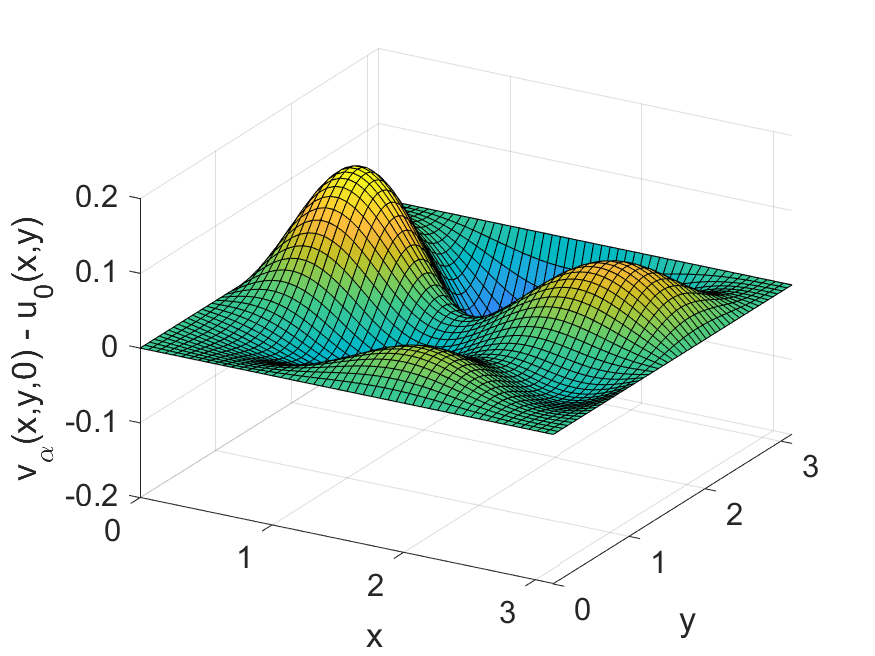} \\
     (c) Reconstruction, \textit{a priori} rule & (d)  Error, \textit{a posteriori} rule \\
         & \\
 \includegraphics[width = 0.45\textwidth]{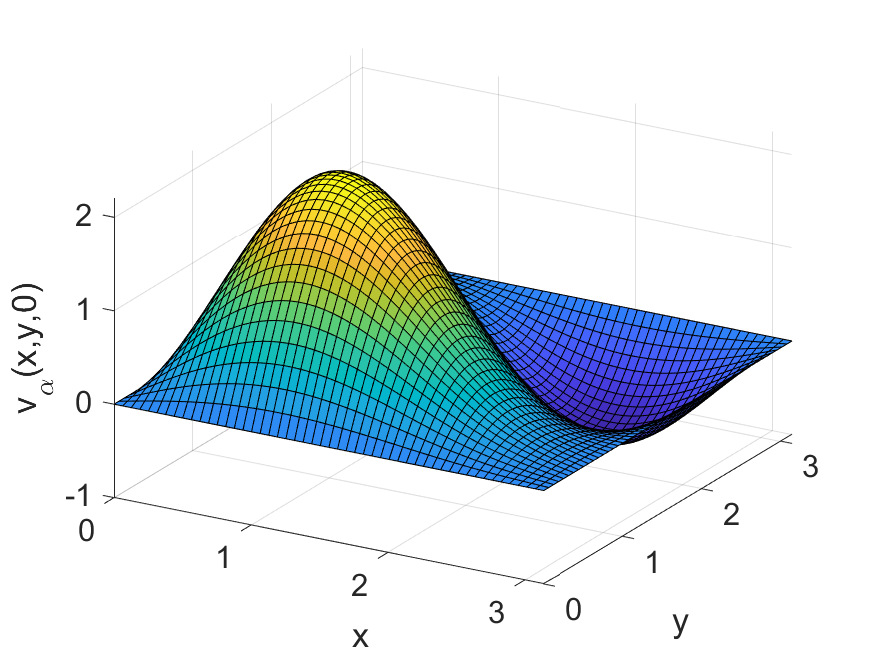} &
  \includegraphics[width = 0.45\textwidth]{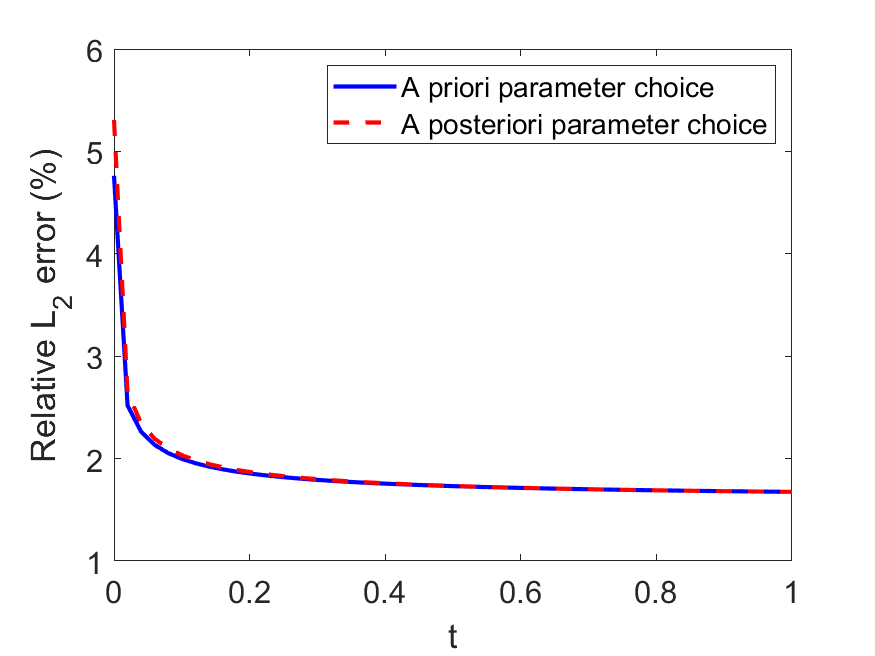} \\
     (e) Reconstruction, \textit{a posteriori} rule & (f)  $L^2$ error distribution in time \\
  \end{tabular}
 \caption{Reconstruction of the initial condition in Example 3  for $\gamma = 1/2$, $p = 3$, $ b= 4$ for noise level = 10\% using the \textit{a priori} parameter choice rule (c) and \textit{a posteriori} parameter choice rule (e). The maximum error is about 10\% and the $L^2$-norm error is about 5\%. Both parameter choice rules produced almost the same results. }
 \label{fig:32}
 \end{figure}

\section{Conclusions}\label{sec:conclusion}

We regularized the backward time-fractional parabolic equations by Sobolev-type equations. We obtained optimal error estimates for the regularized solutions for both \textit{a priori} and \textit{a posteriori} regularization parameter choice rules. The theoretical error estimates were supported by numerical tests for one- and two-dimensional equations.

\bigskip

{\bf Acknowledgments.} The work of N. V. Thang was partly supported by Vietnam
National Foundation for Science and Technology Development
(NAFOSTED) under grant number 101.01-2017.319.


\begin{thebibliography}{99}
\bibitem{mh}M.F.\ Al-Jamal, A backward problem for the time-fractional
diffusion equation. {\it Math.\ Methods Appl.\ Sci}. 40(2017),
2466--2474.

\bibitem{Chadha:2018} A. Chadha,  D. Bahuguna, and  D.N. Pandeym, {Faedo-Galerkin} approximate solutions for nonlocal fractional differential equation of {Sobolev} type. {\it Fract.\ Differ.\ Calc}. 8(2018),
205--222.

\bibitem{ewing}R. E. Ewing, The approximation of certain
parabolic equations backward in time by Sobolev equations, {\it SIAM
J. Math.Anal.}, 6(1975),283--294.


\bibitem{fury}M. A. Fury, Nonautonomous ill-posed evolution problems with strongly elliptic
differential operators, {\it Electron. J. Differential Equations},
92(2013), 1--25.


\bibitem{gaj}H. Gajewski and K. Zacharias, Zur Ruguliarisierung einer
nichtkorrekter Probleme bei Evolutionsgleichungen, {\it J. Math.
Anal. Appl.}, 38(1972), 784--789.


\bibitem{hdt} D.N.\ H\`ao, J. Liu,  N.V.\ Duc, and N.V.\ Thang, Stability results for backward time-fractional parabolic equations,
 {\it Inverse Problems}, 35(2019)
125006 (25pp).


\bibitem{huang}Y. Huang  and Z. Quan, Regularization for a class
of illposed Cauchy problem, {\it Proc. Amer. Math. Soc.}, 133(2005),
3005--3012.

\bibitem{jin}B. Jin and W. Rundell, A tutorial on inverse problems for
anomalous diffusion processes, {\it Inverse Problems}, 31(2015)
035003 (40pp).

\bibitem{long94}N. T. Long  and A. P. N. Dinh, Approximation of
a parabolic non-linear evolution equation backward in time, {\it
Inverse Problems}, 10 (1994),  905--914.

\bibitem{long96}N. T. Long  and A. P. N. Dinh,  Note on a regularization of a
parabolic nonlinear evolution equation backwards in time, {\it
Inverse Problems}, 4(1996), 455--462.


\bibitem{Kilbas} A.A.\ Kilbas, H.M.\ Srivastava, J.J.\ Trujillo, {\it Theory and Applications of Fractional Differential Equations}, Elsevier, 2006.

\bibitem{Liu-Yamamoto-AA} J.J.\ Liu and M.\ Yamamoto, A backward problem for the time-fractional diffusion equation, {\it Appl.\ Anal.} 89(2010), 1769--1788.



\bibitem{padon90}V. Padr\'{o}n, {\it Sobolev regularization of some nonlinear illposed
problems}, PhD thesis. University of Minnensota, Minneapolis, 1990.


\bibitem{padon}V. Padr\'{o}n, Sobolev regularization of a nonlinear
ill-posed parabolic problem as a model for aggregating populations,
{\it Commun. Partial Differential Equations}, 23(1998),457--486.

\bibitem{P} I.\ Podlubny, {\it Fractional Differential Equatins: An Introduction to Fractional Derivatives,
Fractional Differential Equations, to Methods of Their Solution and
Some of Their Applications}, Academic Press, 1999.


\bibitem{renardy}M. Renardy and R. C. Rogers, {\it An Introduction to Partial Differential
Equations}, 2nd Edition, Springer-Verlag, New York Inc. 2004.


\bibitem{Sakamoto-Yamamoto-Jmaa} K.\ Sakamoto and M.\ Yamamoto,
    Initial value/boundary value problems for fractional diffusion-wave
equations and applications to some inverse problems, {\it J. Math. Anal. Appl.} 382(2011), 426--447.

\bibitem{sho}R. E. Showalter, The final value problem for evolution equations,
{\it J. Math. Anal. Appl.} 47 (1974) 563--572.


\bibitem{Liyan Wang-Jijun Liu1} L.\ Wang and J.\ Liu, Data regularization for a backward time-fractional diffusion problem, {\it Comput. Math. Appl.} 64(2012),3613--3626.

\bibitem{Liyan Wang-Jijun Liu} L.\ Wang and J.\ Liu, Total variation regularization for a backward time-fractional diffusion problem, {\it Inverse Problems} 29(2013), 115013, 22pp.

\bibitem{jwang2} J.G.\  Wang, T.\  Wei , Y.B.\ Zhou, Tikhonov regularization method for a backward problem for the
time-fractional diffusion equation,  {\it Appl.\ Math.\ Model.} 37(2013), 8518�-8532.

\bibitem{jwang1} J.G.\ Wang, T.\ Wei, Y.B.\ Zhou, Optimal error bound and
simplified Tikhonov regularization method for a backward problem for
the time-fractional diffusion equation, {\it J. Comput. Appl. Math.}
279(2015), 277--292.

\bibitem{jwang}J.\ G.\ Wang, Y.\ B.\ Zhou, T.\ Wei, A posteriori
regularization parameter choice rule for the quasi-boundary value
method for the backward time-fraction diffusion problem, {\it Appl.\
Math.\ Lett.} 26(2013), 741--747.

\bibitem{twei} T.\ Wei and J.G.\ Wang,  A modified quasi-boundary value method
for the backward time-fractional diffusion problem, {\it ESAIM Math.
Model. Numer. Anal.} 48(2014), 603--621.

\bibitem{Ming Yang- Jijun Liu-ANM}M.\ Yang and J.\ Liu, Solving a final value fractional diffusion problem by boundary condition regularization, {\it Appl. Numer. Math.} 66(2013), 45--58.

\bibitem{yang-liu2015} M.\ Yang and J.\ Liu,  Fourier regularization for a final value time-fractional diffusion problem, {\it Appl.\ Anal.} 94(2015),  1508--1526.

\end{thebibliography}
\end{document}